\pdfoutput=1
\documentclass[11pt,twoside]{amsart}
\title{Degeneracy of the Characteristic Variety}
\author{Abraham D. Smith}
\address{Mathematics Department\\Fordham University\\Bronx, NY 10458-5165}
\email{adsmith@member.ams.org}
\date{October 25, 2014}

\thanks{Thanks to Robert Bryant for a useful conversation at the Fields
Institute in December 2013, in which he suggested that a correct formulation
``should be similar to the appearance of generalized eigenvectors.''  This
suspicion is confirmed by the Main Theorems.  Thanks also 
to Armand Brumer, Ian Morrison, David Swinarski, and Han-Bom Moon for useful
discussions of algebraic geometry. The bulk of the results were proven in
February and March 2014, with exposition developed in June and August 2014.}

\keywords{characteristic variety, Guillemin normal form, eikonal system}

\subjclass[2010]{Primary 58A15, Secondary 35A27}

\usepackage[final]{microtype}

\usepackage{cmtiup}
\usepackage{xcolor}
\usepackage{tikz}
\usetikzlibrary{decorations.pathreplacing,arrows,positioning}
\tikzset{>=stealth}
\usepackage[scr=boondox]{mathalfa}
\usepackage{amssymb}
\usepackage[colorlinks]{hyperref}

\numberwithin{equation}{section}                                 
\newtheorem{mthm}[equation]{Main Theorem}
\newtheorem{thm}[equation]{Theorem}
\newtheorem{conj}[equation]{Conjecture}
\newtheorem{lemma}[equation]{Lemma}
\newtheorem{cor}[equation]{Corollary}
\newtheorem*{cor*}{Corollary}
\theoremstyle{definition}
\newtheorem{defn}[equation]{Definition}
\theoremstyle{remark}
\newtheorem{rmk}[equation]{Remark}

\newcommand{\lhk}{\mathbin{\hbox{\vrule height1.4pt width4pt depth-1pt
\vrule height4pt width0.4pt depth-1pt}}} 

\newcommand{\pair}[1]{\ensuremath{\left\langle #1 \right\rangle}}

\DeclareMathOperator{\Gr}{Gr}
\DeclareMathOperator{\Var}{Var}
\DeclareMathOperator{\C}{\mathcal{C}}
\DeclareMathOperator{\Au}{\mathbf{A}}
\DeclareMathOperator{\Wu}{\mathbf{W}}
\DeclareMathOperator{\End}{End}

\DeclareMathOperator{\codim}{codim}
\DeclareMathOperator{\sat}{sat}
\DeclareMathOperator{\elem}{elem}

\begin{document}
\begin{abstract}
The characteristic variety plays an important role in the analysis of the
solution space of partial differential equations and exterior differential
systems.  This article studies the linear span of this variety, measuring its
dimension via an integrable extension of the original system.  In the PDE case with locally constant characteristic variety,
this extension yields a recursive version of Guillemin normal form, inducing a
sequence of foliations on integral manifolds.
\end{abstract}
\maketitle  \tableofcontents  \newpage

\section{Context}\label{context}
This article investigates the \emph{linear span} of the characteristic variety
of an involutive exterior differential system using established tools of the
discipline, such as eikonal systems, Guillemin normal
form, and integrable extensions.  In particular, we pose the question ``what does degeneracy of the characteristic
variety tell us about solutions of the exterior differential system?'' Despite the
increasingly sophisticated application of commutative algebra to the subject, this simple question has apparently been neglected in the body of late
20th-century work on exterior differential systems.\footnote{
The most important studies of the characteristic variety are 
\cite{Guillemin1968}, 
\cite{Guillemin1970}, \cite{Gabber1981}, and \cite{Malgrangea}, none of which consider degeneracy.
The most thorough single overview is Chapter V of
the book \cite{BCGGG}; however, this chapter's Theorem 3.13 incorrectly equates
$S^\perp$ and $\pair{\Xi}$.
This article arose from an attempt to state and prove a correct
version of that theorem.  The inclusion of that incorrect theorem appears to be
a random error of the drafting and editing process: the theorem is not used
elsewhere in the book, no justification is provided, and a counterexample
appears in the example on Page 276 (Page 235 in the online version).  However,
the incorrect theorem is foreshadowed in a non-technical comment at the bottom
of Page 184 (the middle of Page 159 in the online version).  Based on
conversations with the living authors in 2013, it appears that the error had
gone unreported by other readers.  That a wholly incorrect statement persisted
for so long in a standard reference is strong evidence that the characteristic
variety deserves more careful study.}
Projective varieties are studied over $\mathbb{C}$, and the main theorem can be
put in a weak form as
\begin{cor*}
Suppose an involutive differential ideal $\mathcal{I}$ on a manifold $M$ has
maximal integral elements of projective dimension $n{-}1$, complex
characteristic variety $\Xi$ of projective dimension $\ell{-}1$, and projective
Cauchy system $S$ of dimension $n{-}\nu{-}1$.  Let the complex
linear space $\pair{\Xi}$ have projective dimension $L{-}1$. Then $0 \leq \ell
\leq L \leq \nu \leq n$ and 
\begin{enumerate}
\item $0 =\ell$ if and only if $\mathcal{I}$ is Frobenius;
\item $L=\nu$ if and only if the Guillemin symbol algebras, which are
parametrized by $\Xi$, contain no common nilpotent subalgebra (see Main Theorem~\ref{mainthmS});
\item $\nu=n$ if and only if $(M,\mathcal{I})$ is free of Cauchy retractions;
\item Every ordinary integral manifold is foliated by submanifolds of projective
dimension $n{-}L{-}1$ defined by $\pair{\Xi}=0$ (see Main
Theorem~\ref{foliation}).
\end{enumerate}
The case $L=\nu=n$ shall be called
``elementary,'' which corresponds to $\Xi$ being a non-degenerate variety (see
Main  Theorem~\ref{mainthm0}). 
\label{introcor}
\end{cor*}

A stronger and more precise statement of the results requires significant
conceptual ballast, and Section~\ref{notation} rapidly conveys
notations and definitions for various objects associated with an
exterior differential system.
The terminology here is meant to be familiar and reasonably consistent with
\cite{BCGGG}, diverging only when necessary for a clearer formulation of
results. 
Experts fluent in this language should jump to the Main Theorems in
Section~\ref{mainthms} now.

\section{Notation}
\label{notation}
An \emph{exterior differential system} $(M,\mathcal{I})$ consists of a smooth
manifold $M$ of finite dimension $m$ and an ideal $\mathcal{I}$ in the total
exterior algebra $\Omega^\bullet(M)$ such that $\mathrm{d}\mathcal{I} \subset
\mathcal{I}$ and such that in each degree $p$, the $p$-forms in the ideal,
$\mathcal{I}_{p}=\mathcal{I}\cap \Omega^p(M)$, form a finitely generated
$C^\infty(M)$-module.  For convenience, we assume that $\mathcal{I}_0=0$.
Optionally, we sometimes specify an independence condition as an $n$-form
$\boldsymbol{\omega} \in \Omega^n(M)$ that is not allowed to vanish on solutions.
The category of exterior differential systems includes all smooth systems
of differential equations expressed in local coordinates in jet space.

An \emph{integral element} of $\mathcal{I}$ at $x \in M$ is a linear subspace
$e \subset T_xM$ such that $\varphi|_e = 0$ for all $\varphi \in \mathcal{I}$.
The space of $n$-dimensional integral elements is labeled
$\Var_n(\mathcal{I}) \subset \Gr_n(TM)$.  There is a maximal $n$ for which
$\Var_n(\mathcal{I})$ is locally non-empty, which is the case of interest. 
If an independence condition $\boldsymbol{\omega}$ is specified, we also require $\boldsymbol{\omega}|_e
\neq 0$.

There is an open, dense subset $\Var_n^o(\mathcal{I})\subset \Var_n(\mathcal{I})$ defined as the smooth subbundle of $\Gr_n(TM)$ that is cut out by smooth functions.  These are the
\emph{K\"ahler-ordinary} elements.  A single connected component of
$\Var_n^o(\mathcal{I})$ is called $M^{(1)}$ after $M$ is redefined
to be the open set over which $M^{(1)}$ is a smooth bundle.
Let $s$ denote the dimension of each fiber of the projection $M^{(1)} \to M$,
so $t=n(m-n)-s$ is the corresponding codimension of $T_eM^{(1)}_x$ in
$T_e\Gr_n(T_xM)$.  Such a space $M^{(1)}$ is called the \emph{(ordinary)
prolongation} of $M$, and it admits a prolonged ideal $\mathcal{I}^{(1)}$
generated adding the pullback of $\mathcal{I}$ to the tautological contact system
$\mathcal{J}$ on $\Gr_n(TM)$.

An \emph{integral manifold} of $\mathcal{I}$ is an immersion $\iota:N \to M$
such that $\iota^*(\varphi)=0$ for all $\varphi \in \mathcal{I}$.  If an
independence condition $\boldsymbol{\omega}$ is specified, we require that $\iota^*(\boldsymbol{\omega})
\neq 0$.  That is, a maximal integral manifold is a submanifold all of
whose tangent spaces are maximal integral elements, so $\iota_*(TN) \subset
\Var_n(\mathcal{I})$.  A maximal integral manifold is called
\emph{ordinary} if $\iota_*(TN) \subset M^{(1)}$, in which case the immersion $\iota^{(1)}:N \to
M^{(1)}$ defined by $\iota^{(1)}:y \mapsto \iota_*(T_yN) \in M^{(1)}_{\iota(y)}$ is
called the \emph{prolongation} of $\iota:N\to M$.  The prolonged integral
manifold $\iota^{(1)}:N \to M^{(1)}$ is an integral manifold of the prolonged
system $(M^{(1)}, \mathcal{I}^{(1)})$.
The overall goal is to construct all ordinary integral manifolds of
$(M,\mathcal{I})$ through the careful study of the prolongation $M^{(1)}$. 

Given an integral element $e' \in \Var_{n-1}(TM)$, we consider its
space of integral extensions, called the \emph{polar space}, \[ H(e') = \{ v
~:~
e=e' + \pair{v} \in \Var_{n}(\mathcal{I}) \}\subset TM \] and the
\emph{polar equations} comprise its annihilator,  \[ H^\perp (e') =
\{ e'\lhk \varphi~:~ \varphi \in \mathcal{I}_n  \} \subset T^*M. \] Let
$r(e') = \dim H(e') - \dim e' -1$,
called the \emph{polar rank}, so $\codim H^\perp(e') = n+r$.
Note that $r(e')=-1$ means that $e'$ admits no extensions, and
$r(e')=0$ means that $e'$ admits a unique extension.  Suppose that $e \in
M^{(1)}$, so that $r(e')=0$ for an open set of $e' \in \Gr_{n-1}(e)$.
(It cannot be positive on an open set, for then the dimension $n$ would not be
maximal.)  As the rank of a system of linear equations, $r:\mathbb{P}e^* \to
\mathbb{N}$ is lower semi-continuous on $M^{(1)}$, but it can increase on a Zariski-closed set:
\begin{defn}
For any $e \in M^{(1)}$, the \emph{characteristic variety
of $e$} is 
\begin{equation} 
\Xi_e = \{ \xi \in \mathbb{P}e^*\otimes \mathbb{C} : r(\xi^\perp) > 0 \} \subset
\mathbb{P}e^*\otimes \mathbb{C}.
\end{equation}
\end{defn}
(Throughout, we work with complex varieties unless otherwise noted.) As a projective variety, let $\dim \Xi_e =\ell-1$ and $\deg \Xi_e = s_\ell$;
both are locally constant on $M^{(1)}$.  When $(M,\mathcal{I})$ is involutive (which has many equivalent
definitions; see \cite{BCGGG}), $\ell$ is the Cartan integer and
$s_\ell$ is the last non-zero Cartan character.  If $(M,\mathcal{I})$ is 
analytic and involutive, then the Cartan--K\"ahler theorem guarantees integral
manifolds parameterized by $s_\ell$ functions of $\ell$ variables.

To study $\Xi_e$ simultaneously for all $e \in M^{(1)}$ in an invariant manner,
recall that the Grassmannian space $\Gr_n(TM)$ admits a canonical projective bundle
$\boldsymbol{\gamma}$, which has fiber $\boldsymbol{\gamma}_e = \mathbb{P}e\otimes \mathbb{C}$, and a
canonical dual bundle $\boldsymbol{\gamma}^*$, which has fiber $\boldsymbol{\gamma}^*_e =
\mathbb{P}e^*\otimes \mathbb{C}$.  Since $M^{(1)}$ is a submanifold of
$\Gr_n(TM)$, it admits restricted bundles $V=\boldsymbol{\gamma}|_{M^{(1)}}$ and
$V^*=\boldsymbol{\gamma}^*|_{M^{(1)}}$ with fibers 
\begin{equation}
\begin{split}
V_e &= \mathbb{P}e\otimes \mathbb{C},\text{ and}\\
V^*_e &=\mathbb{P}e^*\otimes \mathbb{C}
\end{split}
\end{equation}
respectively.  
Bases of $V_e^*$ are useful, so let $\mathcal{F}$ denote the right principal $PGL(n)$ bundle over $M^{(1)}$
whose fiber over $e \in M^{(1)}$ is $\mathcal{F}_e = \{ u: V_e
\overset{\sim}{\to} \mathbb{P}^{n-1}\}$.  That is, a basis
$u^1,\ldots,u^n$ of $V^*_e$ is an element $u$ of $\mathcal{F}_e$. 
Note that, for any integral manifold $\iota:N \to M$ of $\mathcal{I}$ with prolongation
$\iota^{(1)}:N \to M^{(1)}$, the pullback bundle $\iota^{(1)*}\mathcal{F} = \{ u \circ
\iota^{(1)}\}$ is the usual (complexified and projectivized) coframe bundle
$\mathcal{F}_N$ over $N$.  

With these canonical bundles in place, $\Xi$ is a global object over
$M^{(1)}$ when considered as a subvariety of $V^*$.  More precisely, define the
\emph{characteristic sheaf of $\mathcal{I}$}, denoted $\mathcal{M}$, as the
sheaf over $V^*$ defined by the homogeneous condition 
that the linear system $H^\perp(\xi^\perp)$ has submaximal rank at $\xi \in
V^*_e$.
The characteristic variety is the support of $\mathcal{M}$.

Another way to see $\Xi_e$ is to view $T_eM^{(1)}_x$ as a
subspace of $e^\perp \otimes e^*$, which is canonically identified with
$T_e\Gr_n(T_x M)$.  Specifically, let 
\begin{equation}
\begin{split}
W_e &= \mathbb{P}e^\perp \otimes \mathbb{C},\text{ and}\\
A_e &= \mathbb{P}T_eM^{(1)}_x \otimes \mathbb{C}.
\end{split}
\end{equation}
The space $A_e$ is called the (complexified and projectivized)
 \emph{tableau} of $\mathcal{I}$,
and it is defined as the kernel of a linear map
$\sigma$, called the \emph{symbol}:
\begin{equation}
\emptyset \to A_e \to W_e \otimes V^*_e \overset{\sigma} \to A_e^\perp \to
\emptyset.
\label{symtab}
\end{equation}
For any hyperplane $\xi^\perp \subset e$, the condition $r(\xi^\perp) > 0$ is equivalent to the condition $\ker \sigma_\xi \neq 
\emptyset$, where $\sigma_\xi:W_e \to A_e^\perp$ is the restricted symbol map
$\sigma_\xi:z \mapsto \sigma(z \otimes \xi)$.  So, the projective variety of
rank-one elements of the tableau, 
$\C_e = \{ z \otimes \xi \in W_e
\otimes V^*_e : \sigma_\xi(z)=0 \}$, is the incidence correspondence for $\ker\sigma$
over $\Xi_e$, as in Figure~\ref{incidencefig}.

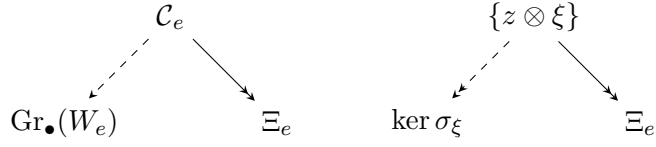
\begin{figure}
\centering\begin{tikzpicture}[node distance=2cm, auto]                          
  \node (C) {$\C_e$};                                                             
  \node[below left of=C] (G) {$\Gr_\bullet(W_e)$};                                
  \node[below right of=C] (Xi) {$\Xi_e$};                                         
  \draw[->>, above right] (C) to node {} (Xi);                                  
  \draw[->, dashed, above left] (C) to node {} (G);                            
  \node[right of=Xi] (G1) {$\ker \sigma_\xi$};                          
  \node[above right of=G1] (C1) {$\{ z \otimes \xi \}$};
  \node[below right of=C1] (Xi1) {$\Xi_e$};                                         
  \draw[->>, above right] (C1) to node {} (Xi1);                                  
  \draw[->>, dashed, above left] (C1) to node {} (G1);                            
\end{tikzpicture}                                                               
\caption{The rank-one variety $\C$ as the incidence correspondence for the
characteristic variety $\Xi$. See \cite{Smith2014a}.}\label{incidencefig}
\end{figure}

\begin{defn}
The \emph{Cauchy retractions}\footnote{
These are typically called Cauchy characteristics, but because this article focuses on the relation between the characteristics $\xi \in \Xi$ and the retractions-n\'ee-characteristics $v \in S$, we hope to avoid confusion through this name change.}  
of $\mathcal{I}$ comprise the subspace $\mathfrak{g} = \{ v \in TM : v \lhk \mathcal{I} \subset \mathcal{I} 
\} \subset TM$.  The ideal generated by $\mathfrak{g}^\perp$ is the
\emph{smallest} Frobenius ideal containing the algebraic generators of
$\mathcal{I}$. (See \cite{Gardner1967} and Section 6.4 of \cite{Ivey2003}.)
Let 
$S_e = \mathbb{P}(e \cap \mathfrak{g})\otimes \mathbb{C} \subset V_e$.
Let $\nu{-}1$ denote the projective rank of the annihilator subbundle $S^\perp
\subset V^*$.
\end{defn}

Let $\pair{\Xi}$ denote the linear subbundle of $V^*$ whose fiber
$\pair{\Xi}_e$ is the span of $\Xi_e$.  Let $L{-}1=\dim\pair{\Xi}_e$.  It is easy to verify that $\pair{\Xi} \subset S^\perp$.
Permanently reserve the following index ranges, where $1 \leq \ell \leq L \leq \nu \leq n
\leq m$:
\begin{equation}
\begin{split}
\lambda,\mu &= 1,\ldots,\ell\\
\varrho,\varsigma &= \phantom{1,\ldots,\ell,}\ell{+1},\ldots\phantom{\ldots,L,L{+}1}\ldots,n\\
i,j &= 1,\ldots\phantom{\ell,\ell{+1},,\ldots}\ldots,L\\
\alpha,\beta &=\phantom{1,\ldots,\ell,\ell{+1},\ldots\ldots,L}L{+}1,\ldots,n\\
k,l &= 1,\ldots\phantom{,\ell,\ell{+1},\ldots\ldots,L,L{+}1}\ldots,n\\
a,b &= \phantom{1,\ldots,\ell,\ell{+1},\ldots\ldots,L,L{+}1,\ldots,n}n{+}1,\ldots,m
\end{split}
\label{resindex}
\end{equation}

If $(u^k)$ is a basis of $V^*_e$ with dual basis $(u_k)$ for $V_e$ and if $(w_a)$ is a basis of $W_e$, then an
element $\pi \in W_e \otimes V^*_e$ may be written as a matrix $\pi = \pi^a_k
(w_a \otimes u^k)$, and the symbol relations $0=\sigma$ defining $A_e$ may be
written as a system of $t$ equations $\{ 0=\sigma^{\tau}(\pi^a_k), \tau=1,\ldots,t\}$.  
For a dense, open subset of these bases, all $s$ generators of
the subspace $A_e$ appear in the matrix $\pi$ according to the Cartan characters, in the first $s_1$ entries of column $1$, the first $s_2$ entries
of column $2$, and so on up to the first $s_\ell$ entries of column $\ell$.  Set 
$s_\varrho=0$ for $\varrho>\ell$.  A basis $(u^k)$ of $V^*_e$ is called
\emph{generic} if the sequence $(s_1, s_2, \ldots,
s_n)$ is lexicographically maximized. A stronger condition is ``$u^k \not\in \Xi_e$ for
all $k$,'' in which case the basis $(u^k)$ of $V^*_e$ is called
\emph{regular}.

\definecolor{rd}{rgb}{0.0, 0.0, 0.9}
\definecolor{io}{rgb}{1.0, 0.65, 0.0}
\definecolor{lb}{rgb}{0.48, 0.45, 0.8}
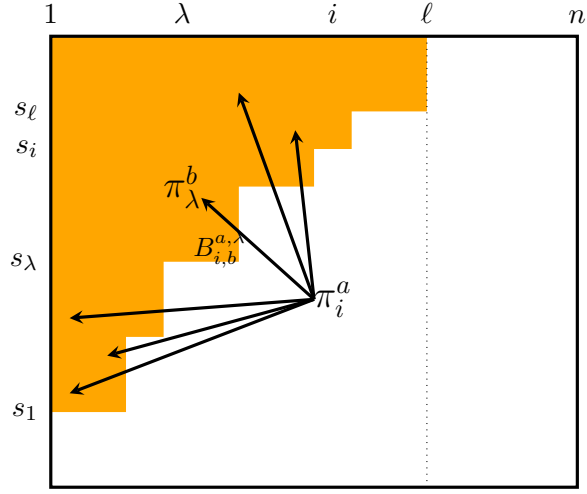
\begin{figure}
\begin{center}
\begin{tikzpicture} [scale=0.5]                                                 
\fill[io] (0,0) -- (10, 0) -- (10,-2) -- ( 9,-2) -- ( 9,-2) -- ( 8,-2) -- ( 8,-3) -- ( 7,-3) -- ( 7,-4) -- ( 6,-4) -- ( 6,-4) -- ( 5,-4) -- ( 5,-6) -- ( 4,-6) -- ( 4,-6) -- ( 3,-6) -- ( 3,-8) -- ( 2,-8) -- ( 2,-10) -- ( 1,-10) -- ( 1,-10) -- ( 0,-10) -- cycle;
\draw[very thick]     (0,0) -- (0,-12) -- (14,-12) -- (14,0) -- cycle;
\draw[dotted]     (10,0) -- (10,-12);
\draw (0,-2) node [left=1pt,black] {$s_\ell$};
\draw (0,-10) node [left=1pt,black] {$s_1$};
\draw (0,-6) node [left=1pt,black] {$s_\lambda$};
\draw (0,-3) node [left=1pt,black] {$s_i$};
\draw (0,0) node [above=1pt,black] {$1$};
\draw (3.5,0) node [above=1pt,black] {$\lambda$};
\draw (7.5,0) node [above=1pt,black] {$i$};
\draw (10,0) node [above=1pt,black] {$\ell$};
\draw (14,0) node [above=1pt,black] {$n$};
\draw  (3.5,-4) node {\Large $\pi^b_\lambda$};
\draw  (7.5,-7) node {\Large $\pi^a_i$};
\draw[very thick, ->,left] (7.0,-7) to node {\small $B^{a,\lambda}_{i,b}$} (4,-4.3);
\draw[very thick, ->] (7.0,-7) -- (5,-1.5);
\draw[very thick, ->] (7.0,-7) -- (0.5,-7.5);
\draw[very thick, ->] (7.0,-7) -- (1.5,-8.5);
\draw[very thick, ->] (7.0,-7) -- (0.5,-9.5);
\draw[very thick, ->] (7.0,-7) -- (6.5,-2.5);
\end{tikzpicture}
\end{center}
\caption{A tableau with Cartan characters $s_1 \geq s_2 \geq
\cdots \geq s_\ell$.  The upper-left shaded entries are
independent generators.  The lower-right entries depend on them via $\pi^a_i =
B^{a,\lambda}_{i,b}\pi^b_\lambda$, summed as in \eqref{symrels}.
See \cite{BCGGG} and \cite{Smith2014a}.}\label{figtab}
\end{figure}

For each $e\in M^{(1)}$, the symbol relations can be reduced as a minimal
system of equations of the form 
\begin{equation}
\Big\{ 0 = \pi^a_k - B^{a,\lambda}_{k,b} \pi^b_\lambda \Big\}_{s_k < a}
\label{symrels}
\end{equation}
where $B^{a,\lambda}_{k,b}=0$ unless 
$\lambda < k$ and  $b \leq s_\lambda$ and $s_k < a $, as discussed
in Chapter IV, \S5 of \cite{BCGGG}.
The symbol relations~\eqref{symrels} can be used to define an element\footnote{
Despite the complicated indexing, \eqref{IB} is just the dual of \eqref{symrels}.  For
example, one often encounters a linear condition like $\pair{\mathrm{d}y^a - p^a_i
\mathrm{d}x^i}$, and describes a solution as $\pair{\frac{\partial}{\partial x^i} +
p^a_i \frac{\partial}{\partial y^a}}$.}
of 
$\End(W_e) \otimes \End(V^*_e)$:
\begin{equation}
\sum_{a \leq s_k} \delta^\lambda_k \delta^a_b (w_a \otimes w^b) \otimes (u^k
\otimes u_\lambda) +   
\sum_{a>s_k} B^{a,\lambda}_{k,b} (w_a \otimes w^b) \otimes (u^k \otimes
u_\lambda).
\label{IB}
\end{equation}
Then, for each $\phi \in V^*_e$, there is a homomorphism $B(\phi):V_e \to
\End(W_e)$ defined by \eqref{IB}.  In Chapter V of \cite{BCGGG}, only the second summand of Equation~\eqref{IB} is
used, and the domain of $B(\varphi)$ is restricted to the annihilator of $\{
u^\lambda \}$, but the identity part is useful for us in Section~\ref{gnf}.
The endomorphism $B(\varphi)(v) \in \End(W_e)$ is most interesting when
restricted to a particular subspace,
\begin{equation}
\Wu_e^1(\varphi) = 
\left\{ z \in W_e : z \otimes \varphi + J^a_\varrho (w_a \otimes u^\varrho) \in
A_e, \text{ for some } J \right\}.
\label{W1e}
\end{equation}
In \cite{Guillemin1968}, Guillemin proved that involutivity implies that
$B(\varphi)(v)|_{\Wu^1_e(\varphi)}$ is an endomorphism of $\Wu^1_e(\varphi)$,
and that these endomorphisms commute for all $v \in V_e$.


The next several definitions are new (or at least, not found in the standard
references), but they allow us to formulate the main theorems clearly.

\begin{defn} An exterior differential system $(M,\mathcal{I})$ is called
\emph{{elementary}} if and only if $\pair{\Xi}_e = V^*_e$ for all $e \in
M^{(1)}$.  \label{defn:elementary}\end{defn}

One can see whether $(M,\mathcal{I})$ is elementary by examining its
characteristic sheaf.  In the language of commutative algebra, recall that an
algebraic ideal admits a saturation
ideal, which is the largest ideal defining the same variety.
The saturation of an ideal is a basic tool in computational algebraic geometry,
using Gr\"obner bases with tools such as Macaulay2. (See \cite{Bayer1993} and Exercise
5.10 on Page 125 of \cite{Hartshorne1977}.) The same terminology applies to a sheaf such as $\mathcal{M}$
with local coordinates parameterizing the fibers of $V^*$.  From that
perspective, ``elementary'' means  $\sat(\mathcal{M})_1=\emptyset$, so
$\sat(\mathcal{M})$ contains no linear functions, meaning that $\Xi$ is defined
only by higher-degree polynomials.  Since $\sat(\mathcal{M})_1$ plays an important
role, we emphasize and relabel it in Definition~\ref{sat1}.

\begin{defn}
Let $X^1$ denote the linear subbundle of $V$ with fiber
\[ X^1_e=\pair{\Xi}^\perp_e = \left\{ v \in \mathbb{P}e :\ v\lhk \xi=0\  \forall \xi \in
\Xi_e \right\} = (\sat \mathcal{M}_e)_1 \subset V_e. \]
\label{sat1}
\end{defn}

Next, we use $X^1$ to construct a new exterior differential system on
$M^{(1)}$.  Let $\omega^1,\ldots,\omega^m$ be a frame on $M$, and lift it to
give 1-forms $\omega^1, \ldots, \omega^m$ on $M^{(1)}$ via the pull-back of the
projection $M^{(1)} \to M$.  (We omit writing the pull-back.)   Fix a
particular element $e \in M^{(1)}$, and suppose that our coframe of $M$ is
generic and adapted so that $\{\omega^a\}$ span $e^\perp$.

Recall that the prolonged system $\mathcal{I}^{(1)}$ on $M^{(1)}$ takes the
form of a restricted contact system:
\begin{equation}
\begin{cases}
0 =h^\tau(P),& \forall \tau=1,\ldots,t\\
0 = \theta^a = \omega^a - P^a_k \omega^k,& \forall
a=n{+}1,\ldots,m 
\end{cases}
\label{proI}
\end{equation}
where the $(m-n)n$ numbers $P^a_k$ provide coordinates of nearby elements in $\Gr_n(TM)$ and the
$t$ functions $h^\tau$ describe the smooth submanifold $M^{(1)} \subset
\Gr_n(TM)$ of dimension $m{+}s$.  Their derivatives $0=\mathrm{d}h^\tau = \frac{\partial
h^\tau}{\partial P}\mathrm{d}P$ provide the symbol 
map $\sigma$ defining the tableau.

In a neighborhood of $e$, we may apply the independence condition
$\boldsymbol{\omega}=\omega^1\wedge\cdots\wedge\omega^n$ and write the degree-2 generators of
$\mathcal{I}^{(1)}$ using the tableau
$0=\sigma(\pi^a_k)$ as
\begin{equation}
\mathrm{d}\theta^a \equiv \pi^a_k\wedge\omega^k = \pi^a_i \wedge\omega^i +
\pi^a_\alpha\wedge\omega^\alpha \mod \{ \theta^b \}
\end{equation}

For each $i=1,\ldots,L$, fix $\xi^i \in \Xi_e$ and extend it to a local section
of $\Xi$ such that $\{\xi^i\}$ forms a basis of
$\pair{\Xi}$ in a neighborhood of $e$.  
Because the coframe $\omega^k$ is generic, it must be that $\xi^i =
H^i_j \omega^j + K^i_\beta \omega^\beta$ for some invertible $L \times L$
matrix $H$.  
Apply a
change of coframe to $M^{(1)}$ depending on $e$ so that $H^i_j\omega^j \mapsto \omega^i$. It can be arranged that the resulting coframe
is still generic.  (A particular method of changing the coframe this way is 
the linear projection described in Section~\ref{gnf}.)
Re-label $P$, 
$K$, and $\pi$ using this new coframe.   Near any $e  \in M^{(1)}$,
consider the system 
\begin{equation}
\begin{cases}
0=h^\tau(P),&\forall \tau=1,\ldots,t\\
0=\theta^a = \omega^a - \left(P^a_\beta - P^a_i K^i_\beta\right) \omega^\beta,&
\forall a=n{+}1,\ldots,m\\
0=\xi^i = \omega^i + K^i_\beta \omega^\beta,& \forall
i=1,\ldots,L 
\end{cases}
\label{elemI}
\end{equation}
Therefore, using the coframe 
$( \xi^i, \omega^\alpha, \theta^a, \cdots)$
on $M^{(1)}$ and the symbol $\sigma(\pi^a_k)=0$, the derivatives of system~\eqref{elemI} take the
form 
\begin{equation}
\begin{cases}
\mathrm{d}
\theta^a \equiv \left( \pi^a_\alpha - \pi^a_j
K^j_\alpha\right)\wedge\omega^\alpha,&  \mod \{ \theta^b, \xi^j\}\\
\mathrm{d}\xi^i \equiv \kappa^i_\alpha\wedge\omega^\alpha,&  \mod \{ \theta^b,
\xi^j\}
\end{cases}
\label{elemI2}
\end{equation}
\begin{defn}
Let $\elem(\mathcal{I})$ denote the linear Pfaffian system defined
locally on $M^{(1)}$ that is generated by Equations~\eqref{elemI} and
\eqref{elemI2} with independence condition 
$\omega^{L+1}\wedge\cdots\wedge\omega^n \neq 0$.
\end{defn}
Note that this system is generally not well-defined on $M$ because the
coefficients $K^i_\beta$ vary with $e \in M^{(1)}$. 
The system $\elem(\mathcal{I})$ is said to \emph{descend} to $M$ if all
vertical vector fields (the kernel of $TM^{(1)} \to TM$) are Cauchy retractions
of $\elem(\mathcal{I})$.  Moreover, the system $\elem(\mathcal{I})$ must be
defined on the complexification of $M^{(1)}$, since $\Xi$ is a complex variety.

Let $\elem^0(\mathcal{I}) = \mathcal{I}$, and recursively define
$\elem^{k}(\mathcal{I})=\elem(\elem^{k-1}(\mathcal{I}))$.

We can now state the main theorems.

\section{Main theorems}\label{mainthms}
\begin{mthm}
Let $(M,\mathcal{I})$ be an involutive exterior differential system with no
Cauchy retractions.  The following are equivalent:
\begin{enumerate}
\item\label{t0span} The ideal $\mathcal{I}$ is elementary, meaning
$\pair{\Xi}_e = V^*_e$ for all $e \in M^{(1)}$;
\item\label{t0sat1} $\left(\sat{\mathcal{M}}\right)_1=\emptyset$;
\item\label{t0elem} The system $\elem(\mathcal{I})$ on $M^{(1)}$ is Frobenius
(in particular, irrelevant);
\item\label{t0desc} The system $\elem(\mathcal{I})$ on $M^{(1)}$ descends to $M$.
\item\label{t0gnf}  If the Guillemin symbol endomorphism $B(\varphi)(v)|_{\Wu^1(\varphi)}$ is nilpotent for all $\varphi$, then $B(\varphi)(v)=0$.
\end{enumerate}
\label{mainthm0}
\end{mthm}

\pagebreak
\begin{mthm}
Let $(M,\mathcal{I})$ be an involutive exterior differential system.   The
following are equivalent:
\begin{enumerate}
\item\label{tSspan} $\pair{\Xi_e} = S^\perp_e$ for all $e \in M^{(1)}$;
\item \label{tSsat} $\left(\sat{\mathcal{M}}\right)_1=S$;
\item\label{tSelem} The system $\elem(\mathcal{I})$ on $M^{(1)}$ is Frobenius;
\item\label{tSdesc} The system $\elem(\mathcal{I})$ on $M^{(1)}$ descends to $M$;
\item\label{tSgnf}  If the Guillemin symbol endomorphism $B(\varphi)(v)|_{\Wu^1(\varphi)}$ is nilpotent for all $\varphi$, then $B(\varphi)(v)=0$.
\end{enumerate}
\label{mainthmS}
\end{mthm}

It is interesting that statements \ref{tSelem}, \ref{tSdesc}, and \ref{tSgnf}
ignore Cauchy retractions entirely.  This suggests that they may be useful when
studying ``intrinsic'' equivalence of Lie pseudogroups in the sense of mutual
coverings and B\"acklund transformations.  
The intrinsic nature of statement
\ref{tSgnf} is not very surprising, but the intrinsic nature of statement
\ref{tSelem} suggests a new invariant of $(M,\mathcal{I})$, which is the subject of the
next corollary.

\begin{cor} 
For any exterior differential system $(M,\mathcal{I})$, there exists 
some $\varepsilon \leq n$ such that the ideal
$\elem^\varepsilon(\mathcal{I})$ is Frobenius.
The minimum such $\varepsilon$ is called the
\emph{elementary depth} of $\mathcal{I}$.
Moreover, for any $e \in M^{(1)}$, there is a flag
\begin{equation} 
V_e = X_e^0 \supset X_e^1 \supset X_e^2 \supset \cdots \supset X_e^\varepsilon =
S_e 
\label{elemflag}
\end{equation}
where $(X_e^k)^\perp$ is the span of the characteristic variety of
$\elem^{k-1}(\mathcal{I})$.
\label{maincor1} 
\end{cor}
In the case that $\mathcal{I}$ is already Frobenius, $\varepsilon=0$, for Frobenius
ideals are identical to their prolongation and have no characteristic variety,
so $(M,\mathcal{I})$ Frobenius trivially implies
$\elem^1(\mathcal{I})=\mathcal{I}^{(1)}=\mathcal{I}=\elem^0(\mathcal{I})$ is
Frobenius.  

The elementary system may be pulled back to maximal ordinary integral
manifolds, and there it is Frobenius, as given by Main Theorem~\ref{foliation}.

\begin{mthm}
Suppose that $(M,\mathcal{I})$ is an involutive exterior differential system.
For every maximal ordinary integral manifold $\iota:N \to M$ and every $y \in N$, there
are unique submanifolds $\Lambda \subset D \subset N$ such that $T_y \Lambda = S_N$
and $T_y D = X^1_N$. 
That is, every ordinary integral element $\iota^{(1)}:N\to M^{(1)}$ is locally foliated by manifolds $D$ integral to
$\elem(\mathcal{I})$, and each such $D \subset N$ is foliated by manifolds
$\Lambda$ integral to $\mathfrak{g}^\perp$.
\label{foliation}
\end{mthm}
The qualifier ``locally'' is required in Main Theorem~\ref{foliation} because
the eikonal system does not guarantee global solutions.
Main Theorem~\ref{foliation} does \emph{not} imply that $\elem(\mathcal{I})$ is
Frobenius as an ideal on $M^{(1)}$, nor does it even imply that
$\elem(\mathcal{I})$ is involutive.  At most, it yields
Corollary~\ref{proelemC}.

\begin{cor}
If $(M,\mathcal{I})$ is an analytic involutive exterior differential system,
then some prolongation of $\elem(\mathcal{I})$ over $M^{(1)}\otimes \mathbb{C}$
is involutive.
\label{proelemC}
\end{cor}

The strongest possible version of Corollary~\ref{proelemC} would be the following
conjecture.
\begin{conj}
Suppose that $(M,\mathcal{I})$ is an analytic involutive exterior differential system,
considered over $\mathbb{C}$.
Then $\elem(\mathcal{I})$ is involutive on $M^{(1)}$, and the integral manifold
$D$ from Main Theorem~\ref{foliation} is ordinary.
\label{invconj}
\end{conj}
As seen in Section~\ref{intext}, this conjecture holds in the case that the
involutive exterior differential system $(M^{(1)},\mathcal{I}^{(1)})$
represents a PDE in local jet-space coordinates such that the span of the
characteristic variety is locally constant.  A general proof of
Conjecture~\ref{invconj} eludes the author in light of significant technical
obstacles discussed in Section~\ref{proext}, but it would imply a beautifully recursive version of Main Theorem~\ref{foliation}.

\begin{mthm}
Suppose that Conjecture~\ref{invconj} holds.
If $(M,\mathcal{I})$ is an analytic involutive exterior differential system
over $\mathbb{C}$, then 
every ordinary integral manifold $N$ of
$(M,\mathcal{I})$ is foliated locally by submanifolds $N \supset D^1 \supset D^2 \supset \cdots
\supset D^\epsilon = \Lambda$ where $TD^k=X^k$.

Moreover, each $X^k$ admits a decomposition 
$X^k=U^{k+1} \oplus Y^{k+1} \oplus X^{k+1}$ where
the characteristic variety of $\elem^{k}(\mathcal{I})$ spans $(U^{k+1} \oplus
Y^{k+1})^*$ and admits a finite branched cover over $(U^{k+1})^*$.
\label{mega-foliation}
\end{mthm}
When it holds, Main Theorem~\ref{mega-foliation} can be seen as a recursive version of
Guillemin normal form, in the sense that the Guillemin symbols of $\elem^k(\mathcal{I})$
form commutative algebras on $(Y^{k+1} + X^{k+1})$ in the usual way (see
Theorem~\ref{thm-gnf1}).

One other important case does not require any recursion.
\begin{cor}
Suppose that $(M,\mathcal{I})$ is involutive and $\ell = n-1$.  (For example,
if it is determined.\footnote{Recall that an exterior differential system is
called \emph{determined} if $\dim A^\perp_e = \dim W_e$ and $\Xi \neq V^*_e$,
equivalently if $\ell=n-1$ and $s_1=s_2=\cdots=s_{n-1} = \dim W_e$, as
discussed in Section 1.4 of \cite{Yang1987}.})
Then exactly one of the following must hold:
\begin{enumerate}
\item\label{determined0} $\ell = L = \nu < n$, in which case
$(M,\mathcal{I})$ admits
Cauchy retractions to an elementary involutive system in dimension
$n-1$;
\item\label{determined1} $\ell = L < \nu = n$, in which case each
maximal ordinary integral manifold locally admits a foliation by
curves annihilated by $\Xi|_N$; 
\item\label{determinedI} $\ell < L = \nu = n$, in which case 
each maximal ordinary integral manifold locally admits a complete system of
characteristic coordinates.
\end{enumerate}
\label{determined}
\end{cor}

The remainder of this article proves these theorems (and a few others) in a
piecemeal manner, first using the eikonal system in Section~\ref{int} to
guarantee that bases adapted to $\pair{\Xi}_e$ can be extended to frames on
$N$, then adapting Guillemin normal form in Section~\ref{gnf} to 
express $X^1$ in terms of the symbol, and finally
exploring the integrable extension $\elem(\mathcal{I})$ in
Section~\ref{intext}.  Sections~\ref{parabolic} and \ref{sec:art} show examples
that suggest future work.

\section{Involutivity of the eikonal system}
\label{int}
Suppose that $\Sigma$ is a sub-bundle of $V^*$ whose fiber $\Sigma_e$ over any
$e \in M^{(1)}$ is a projective variety.  On any ordinary $n$-dimensional
integral manifold $\iota:N\to M$, we have that $\iota_*(TN) \subset M^{(1)}$.
Consider the restricted bundle $\Sigma_N = \Sigma_{\iota_*(TN)}$, which may be
considered via the immersion $\iota$ as a projective sub-variety of $T^*N$.

Now, $T^*N \times \mathbb{R}$ is identical to the jet space $\mathbb{J}^1(N,\mathbb{R})$ and carries
a canonical contact 1-form $\Upsilon$ that may be expressed in local jet
coordinates $(y^1,\ldots,y^n$, $z$, $p_1,\ldots,p_n)$ as $\Upsilon=\mathrm{d}z-p_k
\mathrm{d}y^k$.  Let 
$\psi:\Sigma_N \to T^*N$ denote the inclusion defining $\Sigma_N$. Since each
fiber is a projective variety, $\Sigma_N$ is defined locally by functions
$F^\lambda(y,p)$ that are homogeneous polynomials in $p$.   The \emph{eikonal
system} of $\Sigma$, denoted by
$\mathscr{E}(\Sigma_N)$, is the Pfaffian system on $\Sigma_N \times \mathbb{R}$
that is differentially generated by $\psi^*(\Upsilon)$ with independence
condition $\mathrm{d}y^1 \wedge \cdots \wedge\mathrm{d}y^n$.
The purpose of the eikonal system is to obtain specific results of the following form:
\begin{lemma}
Suppose that $(M,\mathcal{I})$ is involutive, that $\Sigma \subset V^*$ is a
projective variety, that $\iota:N \to M$ is an ordinary integral manifold of
$(M,\mathcal{I})$, and that the eikonal system $(\Sigma_N,\mathscr{E}(\Sigma_N))$
is involutive.  Then, for any $\xi_0$ in the fiber $\Sigma_{N,y}$ over $y$, there is at
least one hypersurface $H\subset N$ such that $(T_yH)^\perp = \ker \xi_0$ and
such that $(T_zH)^\perp \in \Sigma_{N,z}$ for all $z \in H$.    Moreover, 
such hypersurfaces are parameterized according to the Cartan characters of
$\mathscr{E}(\Sigma_N)$.
\label{genint}
\end{lemma} 

For various projective varieties $\Sigma$ that one might choose to study,
establishing the involutivity of $\mathscr{E}(\Sigma_N)$ may be of wildly varying
difficulty.
In the case $\Sigma=S^\perp$, the theorem is nearly trivial:
\begin{thm}
For any ordinary integral manifold $N$, the eikonal system of restricted Cauchy
retractions, $\mathscr{E}(S^\perp_N)$, is involutive with Cartan characters
$s_1=s_2=\cdots=s_\nu=1$.
\label{intcauchy}
\end{thm}
\begin{proof}
The Cauchy retractions $S \subset TM$ are closed under bracket, so
they form an integrable distribution.  That is, $\mathfrak{g}^\perp \subset
T^*M$ is a Frobenius system on $M$.  Therefore, for any integral manifold $\iota:N
\to M$ of $(M,\mathcal{I})$, we have that
$S^\perp_N=\iota^*(\mathfrak{g}^\perp)$ is
a Frobenius system as well.   Therefore, we may choose coordinates
$(y^1,\ldots, y^n)$ on $N$ such that $S^\perp_N \subset T^*N$ is
the span of $\mathrm{d}y^1, \mathrm{d}y^2, \ldots, \mathrm{d}y^\nu$.  In other
words, $\varphi = p_k\mathrm{d}y^k$ is in $S^\perp_N$ if and only of
$p_{\nu+1}=\cdots=p_n=0$, so $S^\perp_N$ is defined by these $n-\nu$
functions, and $TS^\perp_N$ is defined by 
$\mathrm{d}p_{\nu+1}=\cdots=\mathrm{d}p_n=0$. Therefore, the eikonal system 
has generating 2-form   
\begin{equation}
\psi^*(\mathrm{d}\Upsilon)
= - \mathrm{d}p_1\wedge\mathrm{d}y^1 - \cdots -
\mathrm{d}p_\nu\wedge\mathrm{d}y^\nu.
\end{equation}
This is involutive with Cartan characters $s_1=s_2=\cdots=s_\nu=1$.
\end{proof}

In the case $\Sigma=\Xi$, the theorem is very deep and difficult.  It is known as
``the integrability\footnote{Properly, it ought to be called the
\emph{involutivity} of characteristics, since the characteristic hypersurfaces
are unique only in the case that $\mathcal{I}$ has Cartan integer $\ell =1$.}   
 of characteristics,'' as summarized in Theorem~\ref{intchar}.
\begin{thm}[Guillemin--Quillen--Sternberg, Gabber]
For any ordinary integral manifold $N$ of an involutive exterior differential
system $(M,\mathcal{I})$, the eikonal system of the characteristic
variety, $\mathscr{E}(\Xi_N)$, is involutive.  At smooth points
in $\Xi \times \mathbb{R}$, the  Cartan characters are $s_1=s_2=\cdots=s_\ell
=1$.  
\label{intchar}
\end{thm}
Cartan showed many examples of Theorem~\ref{intchar} in \cite{Cartan1911} and
probably at a 1911 lecture ``Sur les caract\'eristiques de certains syst\'emes
d'\'equations aux d\'eriv\'ees partielles'' whose abstract appears immediately
after \cite{Cartan1911} in Volume 2 of his collected works.
The first complete proof in the PDE case appears in \cite{Guillemin1970}, and a
general algebraic proof appears in \cite{Gabber1981}.  Reexaminations of these
proofs appear in \cite{Malgrangea} and Chapter V of \cite{BCGGG}.

For our present purposes, we are concerned with the case of
$\mathscr{E}(\pair{\Xi}_N)$, which one expects to lie neatly between the easy
case of $\mathscr{E}(S^\perp_N)$ and the difficult case of
$\mathscr{E}(\Xi_N)$.    We avoid proving involutivity from scratch, instead
using the difficult case as a crutch, with the following lemma.

\begin{lemma}
If $\mathscr{E}(\Sigma_N)$ is involutive, then $\mathscr{E}(\pair{\Sigma}_N)$ is
involutive. 
\label{spaninv}
\end{lemma}
\begin{proof}
Since we are only concerned with the case $\Sigma=\Xi$, we use notation
consistent with Section~\ref{context}, but no peculiar properties of the
characteristic variety are used.   Let $\Sigma_e$, $\pair{\Sigma}_e$, and $N$ have 
dimension $\ell$, $L$, and $n$ respectively, and recall the index ranges
reserved in Equation~\eqref{resindex}.

Since $\dim\pair{\Sigma}_{N,y}=L$, we may choose   
linearly independent $\xi^1_0, \ldots, \xi^L_0 \in \Sigma_y$ and also (because
$\mathscr{E}(\Sigma_N)$ is involutive) local
extensions $\xi^1, \ldots, \xi^L \in \Sigma_N$ such that $\mathrm{d}\xi^i
\equiv 0 \mod \xi^i$ for each $i=1,\ldots,L$.  That is, we choose $L$
linearly independent characteristic hypersurfaces defined by local
functions $y^i:N \to \mathbb{R}$ such that $\mathrm{d}y^i =
\xi^i$.
Complete $(y^1,\ldots,y^L)$ to a local coordinate system $(y^1,\ldots,y^n)$ on $N$, and
let $(p_1,\ldots,p_n)$ be the
corresponding symplectic coordinates on $T^*_yN$.  Note that completing the coordinate
system is possible because $\mathscr{E}(\mathbb{P}T^*N)$ is itself trivially involutive.

In our chosen coordinates, $\pair{\Sigma}_{N,y}$ is merely the subspace of $T^*_yN$ defined
by the $n{-}L$ functions $0=p_\alpha$, so $T\pair{\Sigma}_N$ is defined by
$\mathrm{d}p_\alpha=0$.  Therefore, the eikonal system of $\pair{\Sigma}_N$ 
has generating 2-form
\begin{equation}
\begin{split}
\mathrm{d}\Upsilon &= 
-\mathrm{d}p_i \wedge \mathrm{d}y^i - \mathrm{d}p_\alpha \wedge \mathrm{d}y^\alpha \\
&\equiv 
-\mathrm{d}p_i \wedge \mathrm{d}y^i  \mod \{ \mathrm{d}p_\alpha \}
\end{split}\end{equation}
This is involutive with Cartan characters $s_1=s_2=\cdots=s_L=1$.
\end{proof}

Using Lemma~\ref{genint} for $\Xi_N$, $\pair{\Xi}_N$, and $S^\perp_N$
sequentially to build a full coordinate system, we obtain:
\begin{cor}
Suppose that $(M,\mathcal{I})$ is an involutive exterior differential system,
and that $N$ is a maximal ordinary integral manifold.
Then $N$ admits a coordinate system $(y^1, \ldots, y^n)$ such that 
$\mathrm{d}y^1,\ldots,\mathrm{d}y^\ell \in \Xi_N$, such that 
$\mathrm{d}y^1,\ldots,\mathrm{d}y^L \in \pair{\Xi}_N$ 
and such that $\mathrm{d}y^1,\ldots,\mathrm{d}y^\nu \in
S^\perp_N$.  For generic smooth points in $\Xi_N$, the choice of such coordinates depends on $\ell$ functions of $\ell$ variables, $L-\ell$ functions of $L$ variables, $\nu-L$
functions of $\nu$ variables, and $n-\nu$ functions of $n$ variables.
\label{nicecoords} 
\end{cor}

\section{Guillemin normal form}\label{gnf} 
Guillemin normal form of the tableau $A$ plays an essential role in the proofs
of the main theorem.  A comment on our approach:  The literature contains two
notable versions of Guillemin normal form.  The first, seen in
\cite{Guillemin1968, Guillemin1970} and discussed Chapter VIII \S6 of
\cite{BCGGG}, is essentially coordinate-free and implies commutativity of the
symbol maps on certain non-characteristic subspaces of $V$
using a linear projection of the characteristic variety.  The second is the
iterative method described in Section 1.1 of \cite{Yang1987}, which explicitly
uses a chosen coframe of $V$, but allows one to state a commutativity
condition on both characteristic and non-characteristic subspaces. 
To state the results most elegantly, and to identify some subtleties, we use
the mixture of these two perspectives that is developed in \cite{Smith2014a}.
See that article for further discussion of the lemmas in this section.


For any $e \in M^{(1)}$, consider the projective space $V^*_e=\mathbb{P}e^*$ of
dimension $n{-}1$.  
Let $X^*_e \subset V^*_e$ be a linear subspace of dimension $n{-}L{-}1$ such
that $\pair{\Xi}_e \cap X^*_e=\emptyset$.  Similarly, we
may choose a linear subspace $Y^*_e \subset \pair{\Xi}_e$ of dimension
$L{-}\ell{-}1$ such that $Y^*_e \cap \Xi_e = \emptyset$.   If $L=\ell$, 
we allow $Y^*_e = \emptyset$.  Let $U^*_e \subset
\pair{\Xi}_e$ be a linear subspace of dimension $\ell{-}1$ such that $U^*_e \cap Y^*_e  =
\emptyset$.    So, $V^*_e$ decomposes\footnote{Here, we are
using $Y^*_e \oplus X^*_e$ as a particularly nice example of a maximal
non-intersecting subspace, which would be called $\Omega$ on Page 379 (Page 324
in the online edition) of \cite{BCGGG}.
While the particular choice of $X^*_e$, $Y^*_e$ and $U^*_e$ is not canonical, the
desired lemmas hold for any such decomposition. One could express the linear
projections in a completely invariant manner using additional language from
commutative algebra, but the notation of bases and frames is useful for us.}
as $U^*_e \oplus Y^*_e \oplus X^*_e$.  The notation is meant to be suggestive,
as equating $X^*_e \cong (X^{1}_e)^*$ is equivalent to splitting the exact sequence $\emptyset \to  \pair{\Xi}_e  \to V^*_e \to X^1_e \to \emptyset$ for $X^1_e = \pair{\Xi}_e^\perp$ as in 
Definition~\ref{sat1}.

Let the covectors $u^1, \ldots, u^\ell$ be a basis for $U^*_e$, let
$u^{\ell+1},\ldots, u^L$ be a basis for $Y_e^*$, and let $u^{L+1},
\ldots, u^n$ be a basis for $X^*_e$, so any $\phi \in V^*_e$ can be decomposed as 
\begin{equation}
\phi = \phi_k u^k = \phi_\lambda u^\lambda + \phi_\varrho
u^\varrho = \phi_i u^i + \phi_\alpha u^\alpha
\end{equation}
using the index ranges reserved in Equation~\eqref{resindex}.
Let $(u_k)$ denote the basis of $V_e$ dual to $(u^k)$, so $u_k(\phi) =
\phi_k$ and $u^k(v)=v^k$ for any $v=v^ku_k \in V_e$.

If $\ell < L$, then $\Xi_e \neq \pair{\Xi}_e$, and one may further assume that
$\Xi_e \cap U_e^*=\emptyset$ and $\Xi_e \cap Y_e^*=\emptyset$, in which case
this basis is also regular, meaning $u^k \not\in \Xi_e$ for all $k$.
However, if $\ell = L$, then $\Xi_e =\pair{\Xi}_e$ as sets, so $Y^*_e =
\emptyset$ and $E^*_e = \pair{\Xi}_e = \Xi_e$.  It is therefore impossible for
this basis to be regular.  There are two ways of proceeding: either perturb
$U^*_e$ by a small angle to be non-characteristic, or take care that the
desired lemmas require genericity but not regularity.  We take the latter approach.

For a dense open subset of the bases $(w_a)$ of $W_e$, the generators of $A_e$ appear
in the first $s_1$ entries of column $1$, the first $s_2$ entries of columns
$2$, et cetera, of the matrix $\pi=\pi^a_k (w_a \otimes u^k)$, so the symbol
relations take the form of Equation~\eqref{symrels}.
Recall that the symbol coefficients define a map
\begin{equation}
B(\varphi)(v): z \mapsto  
\sum_{a \leq s_k} w_a \delta^\lambda_k \delta^a_b z^b v^k \varphi_\lambda +   
\sum_{a>s_k} w_a B^{a,\lambda}_{k,b} z^b v^k \varphi_\lambda.
\tag{\ref{IB} \emph{bis}}
\end{equation}

\begin{lemma}
If $\xi \in \Xi_e$, $v \in V_e$, and $z \in \ker \sigma_\xi \subset
W_e$, then
\begin{equation}
B(\xi)(v)z = \xi(v) z.
\label{eigen}\end{equation}\label{lemma-eigen}
\end{lemma}

Despite the neatness of Lemma~\ref{lemma-eigen}, we do not really want to deal
with $\Xi_e$ directly; rather, it is
better to deal with $U^*_e \cong \mathbb{P}^{\ell-1}$, noting that the linear
projection $\Xi_e \to U^*_e$ is a finite branched cover.  Thus, every $\varphi
\in U^*_e$ represents some finite number of corresponding $\xi \in \Xi_e$.

\begin{figure}
\begin{center}
\begin{tikzpicture} [scale=0.5]                                                 
\fill[io] (0,0) -- (10, 0) -- (10,-2) -- ( 9,-2) -- ( 9,-2) -- ( 8,-2) -- ( 8,-3) -- ( 7,-3) -- ( 7,-4) -- ( 6,-4) -- ( 6,-4) -- ( 5,-4) -- ( 5,-6) -- ( 4,-6) -- ( 4,-6) -- ( 3,-6) -- ( 3,-8) -- ( 2,-8) -- ( 2,-10) -- ( 1,-10) -- ( 1,-10) -- ( 0,-10) -- cycle;
\draw[very thick]     (0,0) -- (0,-12) -- (14,-12) -- (14,0) -- cycle;
\draw[dotted]     (10,0) -- (10,-12);
\draw (0,-2) node [left=1pt,black] {$s_\ell$};
\draw (0,-10) node [left=1pt,black] {$s_1$};
\draw (0,-6) node [left=1pt,black] {$s_\lambda$};
\draw (0,-3) node [left=1pt,black] {$s_i$};
\draw (0,0) node [above=1pt,black] {$1$};
\draw (3.5,0) node [above=1pt,black] {$\lambda$};
\draw (7.5,0) node [above=1pt,black] {$i$};
\draw (10,0) node [above=1pt,black] {$\ell$};
\draw (14,0) node [above=1pt,black] {$n$};
\draw  (3.5,-3.5) node {\small $\Wu^-_\lambda$};
\draw  (7.5,-2) node {\small $\Wu^-_i$};
\draw  (8.5,-4.5) node[rotate=90] {\small $\Wu^+_i \cap \Wu^-_\lambda$};
\draw  (8.5,-9.5) node[rotate=90] {\small $\Wu^+_i \cap \Wu^+_\lambda$};
\draw[very thick] (7,-6) -- (7,-12) -- (8,-12) -- (8,-6) -- cycle;
\draw[very thick] (3,0) -- (3,-6) -- (4,-6) -- (4,0) -- cycle;
\draw[very thick] (7,0) -- (7,-6) -- (8,-6) -- (8,0) -- cycle;
\draw[dotted] (7,-3) -- (8,-3);
\draw[very thick, ->] (4.0,-3.5) to node[above] {$0$} (6.9,-2);
\draw[very thick, ->] (4.0,-3.5) to node[below] {$B^\lambda_i$ } (6.9,-4.5);
\draw[very thick, ->] (4.0,-3.5) to node[below] {$0$} (6.9,-9);
\end{tikzpicture}
\end{center}
\caption{The map $B^\lambda_i$ for a tableau satisfying condition \ref{endov}
of Theorem~\ref{thm-gnf+}.  See \cite{Smith2014a}.}
\label{figB}
\end{figure}

For each basis element $u^k$ of $V^*_e$, let
\begin{equation}
\begin{split}
\Wu_e^-(u^k) &= \left\{ z =w_a z^a : z^a = 0\ \forall\ a \leq s_k \right\} \\
\Wu_e^+(u^k) &= \left\{ z =w_a z^a : z^a = 0\ \forall\ a > s_k \right\}
\end{split}
\end{equation}
So that $W_e =  \Wu_e^-(u^k) \oplus \Wu_e^+(u^k)$ and $\Wu_e^-(u^1) \supset
\Wu_e^-(u^2) \supset \cdots \supset \Wu_e^-(u^n)$ because $s_1 \geq s_2
\cdots \geq s_n$.  Of course, for $\varrho > \ell$, we have
$\Wu_e^-(u^\varrho) =
\emptyset$.  For each $\lambda$, consider also the subspace
\begin{equation}
\Au_e^-(u^\lambda) = \left\{ \pi = B(u^\lambda)(\cdot)z,\ z \in \Wu_e^-(u^\lambda)\right\} \subset
A_e
\end{equation}
The symbol relations \eqref{symrels} imply that the coefficients $\pi^a_k$ of
$\pi \in \Au_e(u^\lambda)$ are
determined uniquely by the choice of $z \in \Wu_e^-(u^\lambda)$, so
$\Au_e^-(u^\lambda)$ and $\Wu_e^-(u^\lambda)$ are isomorphic
via the projection onto the $u^\lambda$ column. 

Using this basis and isomorphism, there is a decomposition
\begin{equation}
A_e = \bigoplus_{\lambda=1}^\ell \Au_e^-(u^\lambda) \cong \bigoplus_{\lambda=1}^\ell
\Wu_e^-(u^\lambda).
\label{Adecomp}
\end{equation}
Specifically, if $\pi = \pi^a_k (w_a \otimes u^k) \in A_e$, then let 
\begin{equation}
z_\lambda = \sum_a z^a_\lambda w_a \in W,~\text{for}~ 
z^a_\lambda = 
\begin{cases}
\pi^a_\lambda,& a \leq s_\lambda\ \&\ \lambda \leq \ell\\
0, & \text{otherwise}.
\end{cases}
\label{zdecomp}
\end{equation}
So, the decomposition~\eqref{Adecomp} yields
\begin{equation}
\pi = \sum_\lambda \pi_\lambda = \sum_\lambda B(u^\lambda)(\cdot)z_\lambda.
\label{pidecomp}
\end{equation}
Since $\dim \Wu_e^-(u^\lambda) = s_\lambda$, this is a more precise version of the statement
that, for a generic flag, the tableau matrix has $s_1$ generators in the first
column, $s_2$ in the second column, and so on until the final $s_\ell$
generators in the $\ell$ column.  

The complete linear and quadratic conditions of involutivity are provided by
Theorem~\ref{thm-gnf+}, which is an adaptation of the construction described in
Chapter 1 of \cite{Yang1987} and thus a re-expression of Guillemin normal
form. Compare it to Theorem 7.1 in \cite{Guillemin1968}.
\begin{thm}[Involutivity Criteria]
Let $A$ denote an tableau given in a generic basis of $V^*$ by
with symbol relations \eqref{symrels}, as in Figure~\ref{figtab}.  Write
$B^\lambda_k$ for $B(u^\lambda)(u_k)$. 
The tableau $A$ is involutive if and only if  there exists a basis of $W$ such
that
\begin{enumerate}
\item\label{endov} $B^{a,\lambda}_{k,b}=0$ for all $a > s_\lambda$; 
\item\label{minor} $\left(B^\lambda_lB^\mu_k - B^\lambda_kB^\mu_l\right)^a_b=0$ 
for all $b$, all $\lambda < l < k$ and $\lambda \leq \mu < k$, and all $a>
    s_l$.
\end{enumerate}
In particular, $B(u^\lambda)(v)$ is an endomorphism of $\Wu^-(u^\lambda)$ such
that 
for all $v, \tilde{v} \in (U_e^*)^\perp$,
\[[B(u^\lambda)(v), B(u^\lambda)(\tilde{v})]=0.\]
\label{thm-gnf+}
\end{thm}
Because its notational intricacies are useless for the Main Theorems here, we remove
the discussion of Theorem~\ref{thm-gnf+} to a separate article,
\cite{Smith2014a}.  
Also, compare Corollary \ref{cor:restrict} to Theorem A in \cite{Guillemin1968}.

\begin{cor}[Guillemin]
If $A$ is involutive, then $A|_U$ (the projection of $A$ to $W \otimes U^*$) is involutive. 
\label{cor:restrict}
\end{cor}

The map $B(\varphi)$ makes sense for any $\varphi \in U^*_e$, not just the
basis elements, and the spaces $\Wu^-(u^\lambda)$ can be generalized for any
$\varphi \in U^*_e$ in the following way:
Let $\varphi= \varphi_\lambda u^\lambda$, and let
$\underline{\lambda} = \min\{ \lambda : \varphi_\lambda \neq 0\}$.
Define the space 
\[\Wu^-_e(\varphi) =  \Wu^-_e(u^{\underline{\lambda}}).\]
Then condition~\ref{endov} of Theorem~\ref{thm-gnf+} reveals  $B(\varphi) = \sum_\lambda \varphi_\lambda
B(u^\lambda)$, so involutivity implies that $B(\varphi)(v)$ is an endomorphism
of $\Wu_e^-(\varphi)$; however, the commutativity property is more subtle
because of the ordering of condition~\ref{minor}.

Recall the space $\Wu_e^1(\varphi)$ from \eqref{W1e} studied by Guillemin. The spaces
$\Wu_e^-(\varphi)$ and $\Wu_e^1(\varphi)$ have the following relationship.
\begin{lemma}
For any $\varphi \in U_e^*$, 
\[ 
\Wu^1_e(\varphi) = \left\{ 
z \in \Wu_e^-(\varphi) ~:~ 
\left( \sum_\lambda \varphi_\lambda B^\lambda_\mu z\right)^b = \varphi_\mu
z^b,\ \forall a > s_\mu,\ 
 \forall \mu
\leq \ell 
\right\}.\]
\label{Wup}
\end{lemma}

\begin{thm}[Guillemin]
For every $\varphi \in U^*_e$ and $v \in V_e$, the restricted homomorphism
$B(\varphi)(v)|_{\Wu_e^1(\varphi)}$ is an endomorphism of $\Wu_e^1(\varphi)$, with
\begin{equation}
B(\varphi)(v)z  =  (\varphi_\lambda v^\lambda) z+ (J^a_\varrho v^\varrho)w_a  =
\varphi(v)z + J(v) = 
\pi v,
\label{gnfform}
\end{equation}
where $\pi = B(u^\lambda)(\cdot)z$.
Moreover, for all $v, \tilde{v} \in V_e$, 
\begin{equation} \left[ 
B(\varphi)(v), 
B(\varphi)(\tilde{v})\right]\Big|_{\Wu^1_e(\varphi)}
=0.
\end{equation}
\label{thm-gnf1}
\end{thm}
One important distinction between Theorems~\ref{thm-gnf+} and \ref{thm-gnf1} is the space $\Wu^-_e(u^\lambda)$ versus $\Wu^1_e(u^\lambda)$.
 Note also that the usual statement of this theorem, as in Proposition 6.3 in Chapter VIII of \cite{BCGGG} and Lemma 4.1 \cite{Guillemin1968} restricts
$v,\tilde{v}$ to the subspace $(U_e^*)^\perp \cong Y_e \oplus X_e$, but this is
unnecessary because of our inclusion of the identity term in \eqref{IB}.

Theorems~\ref{thm-gnf+} and \ref{thm-gnf1} allow a converse of Lemma~\ref{lemma-eigen}
in the form of Corollary~\ref{cor:backeigen}.  
\begin{cor}
Suppose that $(M,\mathcal{I})$ is an involutive exterior differential system.
Fix $\varphi \in U^*_e$ and suppose that $z \in \Wu_e^-(\varphi)$ such that $z$ is
an eigenvector of $B(\varphi)(v)$ for every $v \in V_e$.  Then there
is a $\xi \in \Xi_e$ over $\varphi \in U^*_e$ such that $z  \in
\Wu^1_e(\varphi)$, so $z \otimes \xi \in A_e$.
\label{cor:backeigen}
\end{cor}
Corollary~\ref{cor:backeigen} is a bit more subtle than it might first appear.  It is similar to
the construction in the usual proof of Theorem~\ref{thm-gnf1}, 
but that construction requires $z \in \Wu_e^1(\varphi)$ \emph{a priori}.
The key is Lemma~\ref{Wup}.  See \cite{Smith2014a} for details.
Corollary~\ref{cor:backeigen} deserves a warning:  The specification of $\xi$
over $\varphi$ is not unique, as the variety $\Xi$ may have multiple components
and multiplicity.

\begin{lemma}
Suppose that $(M,\mathcal{I})$ is an involutive exterior differential system
with a coframe $u$ on $V$ as described above.
For any $v \in V_e$, the following are equivalent:
\begin{enumerate}
\item\label{l1}  $v \in X^1_e$;
\item\label{l2}  $v^i = u^i(v) = 0$ for all $i=1,\ldots,L$; 
\item\label{l3}  $B(\varphi)(v)|_{\Wu^1_e(\varphi)}$ is nilpotent (possibly
trivial) for all $\varphi \in U_e^*$; 
\end{enumerate}
\label{gnfnilp}
\end{lemma}
\begin{proof}
Recall that $X^1_e = \pair{\Xi}_e^\perp = (U^*_e \oplus Y^*_e)^\perp$.  The
equivalence of statements \ref{l1} and \ref{l2} is immediate in our chosen
basis for $V^*_e$.  

Fix $v \in X^1_e$, and suppose that $\zeta_\varphi(v)$ is an eigenvalue of
$B(\varphi)(v)|_{\Wu^1_e(\varphi)}$ for some $\varphi \in U^*_e$. The
commutativity property of Theorem~\ref{thm-gnf1} holds, so the eigenspace of $\zeta_\varphi(v)$ contains an
eigenvector $z$ that is shared among $\{
    B(\varphi)(\tilde{v})|_{\Wu^1_e(\varphi)} ~:~ \tilde{v} \in
V_e\}$.  Therefore, Equation~\eqref{eigen} holds, and $\zeta_\varphi(v)z =
\xi(v)z$.  By the assumption that $v \in X^1_e = \pair{\Xi}^\perp_e$, we have
$\xi(v)=0$, so the corresponding eigenvalue $\zeta_\xi(v)$ is zero.  

Conversely, choose $v \in V_e$ such that  
$B(\varphi)(v)|_{\Wu^1_e(\varphi)}$ is nilpotent for
all $\varphi \in U^*_e$ representing $\xi \in \Xi$.  Then every eigenvalue of
$B(\varphi)(v)|_{\Wu^1_e(\varphi)}$ is zero. 
Fixing a particular $\varphi$, if $z$ is a mutual eigenvector 
of $\{ B(\phi)(\tilde{v})|_{\Wu^1_e(\varphi)}\ :\ \tilde{v} \in V_e\}$, then $\xi(v)=0$ for all
$\xi \in \Xi_e$ over $\varphi \in U^*_e$.  Since this holds for all $\varphi
\in U^*_e$, we have $v \in \pair{\Xi}^\perp_e = X^1_e$.
\end{proof}

\begin{lemma}
Suppose that $(M,\mathcal{I})$ is an exterior differential system 
equipped with a basis $(u^k)$ of $V^*_e$ and $(w_a)$ of $W_e$ such that the
coefficients $B^{a,\lambda}_{k,b}$ describing $A$ satisfy condition \ref{endov}
of Theorem~\ref{thm-gnf+}.  The following are equivalent:
\begin{enumerate}
\item\label{k1} $v \in S_e$;  
\item\label{k2} $B(\varphi)(v)$ is the trivial endomorphism
for all $\varphi \in U_e^*$; and 
\item\label{k3} $B(\varphi)(v)|_{\Wu^1(\varphi)}$ is the trivial endomorphism for all
$\varphi \in U_e^*$.
\end{enumerate}
\label{gnftriv}
\end{lemma}
\begin{proof}
Now, $v \in S_e$ if and only if $\pi v = 0$ for all $\pi \in A_e$.  
The decomposition \eqref{pidecomp} means this is equivalent to $\pi v =0$ for all
$\pi \in \Au^-_e(u^\lambda)$ for all $\lambda$.  By the isomorphism $\Au_e^-(u^\lambda)
\cong \Wu^-_e(u^\lambda)$, this is equivalent to $B(u^\lambda)(\cdot)z =0$ for all
$z \in \Wu^-_e(u^\lambda)$ for all $\lambda$, which is clearly equivalent to
$B(\varphi)(\cdot)z =0$ for all $z \in \Wu^-_e(\varphi)$ for all $\varphi \in
U^*_e$.   Hence, \ref{k1} and \ref{k2} are equivalent.
Moreover, \ref{k2} implies \ref{k3}, as $\Wu^1_e(\varphi) \subset
\Wu^-_e(\varphi)$.

Suppose \ref{k3} holds for $v$.  
Note that $\Wu^-_e(u^\ell) = \Wu^1(u^\ell)$, so $B(u^\ell)(v)=0$.  If the
Cartan characters are all equal, $s_1=s_2=\cdots=s_\ell$, then the claim \ref{k2} follows trivially.  Therefore,
suppose that $\lambda$ is maximal such that $s_\lambda > s_\ell$.  We 
consider the symbol endomorphisms 
$B(u^\lambda)(u_\ell)$ and $B(u^\lambda)(v)$.
Using $0 < s_\ell  < s_\lambda$, we may consider the
following block-decomposition of $B(u^\lambda)(u_\ell)$:
\begin{equation}
B(u^\lambda)(u_\ell)=
\begin{bmatrix}
0 & 0 \\
C & D 
\end{bmatrix}.
\end{equation}
(The $C,D$ here are merely block matrices, not the objects previously
labeled by those glyphs.)   Note that $z \in \Wu^1(u^\lambda)$ implies that 
$z \in \ker B(u^\lambda)(u_\ell) = \ker(C,~D)$ by Lemma~\ref{Wup}, so the assumption \ref{k3}
implies that $z \in B(u^\lambda)(v)$.  
Now, for any $\varphi = u^\lambda + \tau u^\ell$ in the $\mathbb{P}^1$ spanned
by $u^\lambda$ and $u^\ell$,  we apply Lemma~\ref{Wup} similarly.
In particular, $z \in \ker(C,~D-\tau I)$ implies  $z \in \ker B(u^\lambda)(v)$.
If $C=0$ and $D=0$, then $\Wu^1(u^\lambda) = \Wu^-(u^\lambda)$, so
$B(u^\lambda)(v)=0$ by assumption.  If $C$ or $D$ is non-zero, then varying $\tau$ makes the
kernel of $B(u^\lambda)(v)$ span all of $\Wu^-(u^\lambda)$, so $B(u^\lambda)(v)=0$.  

Repeat this argument, decreasing $\lambda$ until $B(u^1)(v)=0$.  Hence,
\ref{k2} holds.
\end{proof}

The fact that \ref{k3} implies \ref{k2} is actually the key to Main
Theorems \ref{mainthm0} and \ref{mainthmS}; without it, 
the condition of Lemma~\ref{gnfnilp} regarding $\Wu^1_e(\varphi)$ and the
condition of Lemma~\ref{gnftriv} regarding $\Wu^-_e(\varphi)$ are incomparable.

Finally, the choice of basis $(u^1, \ldots, u^n)$ just described in a single
fiber $V^*_e$ may be extended to a local section $u:M^{(1)}\to \mathcal{F}$.
There is still some freedom in selecting the basis $(u^k)$, as there is always
freedom in choosing complementary subspaces and sections of exact sequences.
There is also the usual freedom in extending a particular basis $(u^k)$ of
$V^*_e$ to a local section $u:M^{(1)}\to \mathcal{F}$.  In any case, the coframe is
generic and adapted to $V \supset X^1 \supset S$.

\section{Elementary extension}\label{intext}
In this section, we study the ideal $\elem(\mathcal{I})$ and prove
Main Theorems~\ref{mainthm0}, \ref{mainthmS}, and \ref{foliation}.
The construction of $\elem(\mathcal{I})$ is similar to the notion of an integrable extension as in \cite{Bryant} and Definition 6.5.3
of \cite{Ivey2003}.

Let $\varpi:M^{(1)} \to M$ denote the bundle projection.
We have established a $\mathbb{C}^{m+s}$-valued coframe of $M^{(1)}$
comprised of 
\begin{equation}
(u^i)_{i=1,\ldots,L},\ (u^\alpha)_{\lambda=L+1,\ldots,n},\
(\theta^a)_{a=n+1,\ldots,m},\ \text{and}\ (\pi^a_\lambda)_{a \leq
s_\lambda}.
\label{ucoframe}
\end{equation}
So, $\mathrm{d}u^i \equiv \eta^i_\alpha \wedge u^\alpha \mod \{\theta^b, u^j\}$ for some
forms $\eta^i_\alpha$ that may be written explicitly as 
\begin{equation}
\eta^i_\alpha  \equiv H^i_{\alpha,\beta} u^\beta + \sum_{b \leq s_\mu}
H^{i,\mu}_{\alpha,b} \pi^b_\mu \mod \{ \theta^b, u^j\},
\end{equation}
where the $H$-coefficients are determined by the choice of coframe.

With respect to this coframe, the \emph{complexified}
prolongation system $\mathcal{I}^{(1)}_\mathbb{C}$ on $M^{(1)}$ is generated by
\begin{equation}
\begin{cases}
\theta^a,\\
\mathrm{d}\theta^a \equiv \pi^a_i \wedge u^i + \pi^a_\alpha \wedge
u^\alpha \mod \{ \theta^b\} 
\end{cases}
\label{proC}
\end{equation}
with independence condition $u^1 \wedge \cdots \wedge u^n \neq 0$.
The elementary system $\elem(\mathcal{I}) = \mathcal{I}^{(1)}_\mathbb{C} + \pair{\Xi}$, whose definition implicitly requires complexification, is generated as 
\begin{equation}
\begin{cases}
\theta^a,\\
u^i, \\
\mathrm{d}\theta^a \equiv \pi^a_i \wedge u^i + \pi^a_\alpha \wedge
u^\alpha \mod \{ \theta^b \},\\
\mathrm{d}u^i \equiv \phantom{\pi^a_i \wedge u^i +\  }\ \eta^i_\alpha \wedge u^\alpha \mod \{ \theta^b, u^j\} \\
\end{cases}
\label{elemNew}
\end{equation}
with independence condition $u^{L+1} \wedge \cdots \wedge u^n\neq 0$.
This is the same system described casually in Section~\ref{context}, but now
our coframe of $M^{(1)}$ is adapted to the problem.  See
Figure~\ref{figsubtab}.

\begin{figure}
\centering\begin{tikzpicture}
\draw (-0.5,1.5) node [black] {\Large $A=$};
\draw[very thick] (0,2) -- (3.5,2) -- (3.5,1) -- (0,1) -- cycle;
\draw (0.65,1.5) node [black] {\Large $\pi^a_i$};
\draw (2.5, 1.5) node [black] {\Large $\pi^a_\alpha$};
\draw[thin,dotted] (0,0) -- (3.5,0) -- (3.5,-1) -- (0,-1) -- cycle;
\draw[very thick] (1.5,0) -- (3.5,0) -- (3.5,-2.5) -- (1.5,-2.5) -- cycle;
\draw (0.65,-0.5) node [black] {\Large $\pi^a_i$};
\draw (2.5,-0.5) node [black] {\Large $\pi^a_\alpha$};
\draw (2.5,-1.75) node [black] {\Large $\eta^i_\alpha$};
\draw[thin,dotted] (5,0) -- (8.5,0) -- (8.5,-1) -- (5,-1) -- cycle;
\draw[thin,dotted] (6.5,0) -- (8.5,0) -- (8.5,-2.5) -- (6.5,-2.5) -- cycle;
\draw[thin,dotted] (7,0) --   (8.5,0) --   (8.5,-3) -- (7,-3) -- cycle;
\draw[very thick]  (8,0) --   (8.5,0) --   (8.5,-4) -- (8,-4) -- cycle;
\end{tikzpicture}                                                               
\caption{The tableaux of $\mathcal{I}$, $\elem(\mathcal{I})$, and so on.
Compare to Equations~\eqref{proC} and \eqref{elemNew}.}\label{figsubtab}
\end{figure}
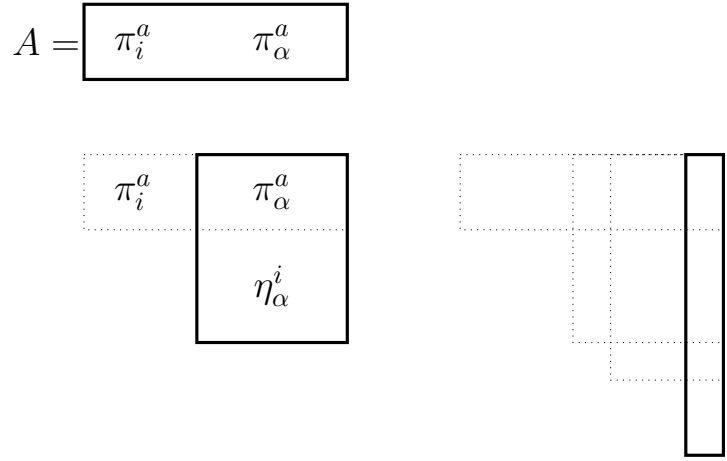

If $\elem(\mathcal{I})$ were itself involutive, then the
decomposition  $V^* = U^* \oplus Y^* \oplus X^*$ could be repeated
for $\elem(\mathcal{I})$.  However, a proof of 
Conjecture~\ref{invconj}  eludes the author, one obstruction being Conjecture~\ref{rankconj}, discussed
below.
Instead, we can prove a slightly weaker version: 
\begin{lemma}
Suppose that $(M,\mathcal{I})$ is an involutive exterior differential system.
Then the system $\elem(\mathcal{I})$ on $M^{(1)}$ admits a smooth family of
maximal integral manifolds of dimension $n{-}L$. 
\label{proelem}
\end{lemma}
\begin{proof}[Proof of Lemma~\ref{proelem} and Main Theorem~\ref{foliation}]
Suppose that $\elem(\mathcal{I})$ is not Frobenius, for the claim is trivial in
that case.
Because $\elem(\mathcal{I})$ contains the involutive ideal $\mathcal{I}^{(1)}$,
we have \[\Var_{k}(\elem(\mathcal{I})) \subset \Var_k(\mathcal{I}^{(1)})\] for
all $k$, so the maximal dimension of ordinary integral elements of $\elem(\mathcal{I})$
cannot be greater than the maximal dimension for $\mathcal{I}^{(1)}$, namely
$n$.   Moreover, since $\elem(\mathcal{I})$ contains $L$ additional generating
1-forms that are independent of $\mathcal{I}^{(1)}$, the maximal dimension of
integral elements of $\elem(\mathcal{I})$ is at most $n-L$. 
Using the independence condition $u^{L+1}\wedge\cdots\wedge u^n \neq 0$, the maximal integral elements $f \in \Var_{n-L}(\elem(\mathcal{I}))$ may be written as 
\begin{equation}
f = \pair{
\pi^a_\lambda - \sum_{a \leq s_\lambda} Q^{a}_{\lambda,\alpha}u^\alpha
}^\perp \subset T_eM^{(1)},
\end{equation}
where the coefficients $Q^{a}_{\lambda,\alpha}$ define $Q \in A_e \otimes
X^{*}_e$ and are subject to the 2-form conditions from \eqref{elemNew},
\begin{equation}
\begin{cases}
\displaystyle{\sum_{b \leq s_\lambda}
B^{a,\lambda}_{\alpha,b}Q^{b}_{\lambda,\beta} =
\sum_{b \leq s_\lambda}B^{a,\lambda}_{\beta,b}Q^{b}_{\lambda,\alpha},}&\\
\displaystyle{H^{i}_{\alpha,\beta} + \sum_{b\leq s_\lambda}H^{i,\lambda}_{\alpha,b}Q^{b}_{\lambda,\beta}=
H^{i}_{\beta,\alpha} + \sum_{b\leq
s_\lambda}H^{i,\lambda}_{\beta,b}Q^{b}_{\lambda,\alpha},} &\forall
\alpha,\beta.
\end{cases}
\label{elem1condition}
\end{equation}
Let $E_e$ denote the subspace of $A_e \otimes X^{*}_e$ defined by the
condition~\eqref{elem1condition}.   The bundle $E$ over $M^{(1)}$ is the tableau
of $\elem(\mathcal{I})$, which is discussed further in Lemma~\ref{dXE}.

We can construct a smooth family of maximal integral manifolds in the following way:  

Fix $e \in M^{(1)}$ and consider the family of ordinary integral manifolds
$\iota:N \to M$, with $y \in N$ and $\iota_*(T_yN) = e$.  Cartan's test for
involutivity of $\mathcal{I}$ guarantees that this family is smooth,
parameterized by $s_\ell$ functions of $\ell$ variables.  For each such $N$,
choose $L$ independent elements of $\Xi_{N,y} \subset T^*_yN$ and use the
eikonal system of $\Xi_N$ to build $L$ independent characteristic hypersurfaces
through $y \in N$.  The intersection of these hypersurfaces is a submanifold $D
\subset N$ of dimension $n-L$.  The submanifold $D$ is unique in the sense that
it does not depend on the particular choice of $L$ characteristic
hypersurfaces, because $\xi|_{TD}=0$ for all $\xi \in \Xi_N$.  Of course,
$\xi|_{TD}=0$ also implies that $\iota^{(1)}|_D:D \to M^{(1)}$ is an integral manifold of $\elem(\mathcal{I})$
through $X^1_e$.

To be explicit, suppose $\hat{e}=\iota^{(1)}_*(T_yN) \subset T_eM^{(1)}$ is given as 
\begin{equation}
\hat{e} = 
\pair{
    u_i + \sum_{a \leq s_\lambda} P^{a}_{\lambda,i}\pi_a^\lambda,\  
    u_\alpha + \sum_{a \leq s_\lambda} P^{a}_{\lambda,\alpha}\pi_a^\lambda }
\end{equation}
where the coefficients define a section $P:N \to A \otimes V^*$.
The subspace $\iota^{(1)}_*(T_yD)$ satisfies $\theta^a =0$ and
$\mathrm{d}\theta^a =0$ because $N$ is integral to $\mathcal{I}^{(1)}$.
This implies the first condition in Equation~\eqref{elem1condition} is
satisfied for $Q^a_{\lambda,\alpha} = P^a_{\lambda,\alpha}$.  (Compare to
Lemma~\ref{dXE}.)
Note that Lemma~\ref{spaninv} implies that $\mathrm{d}u^i \wedge u^i$ pulls back to $0$
on $N$, as our choice of coframe $u:N \to \mathcal{F}_N$ determines 
a function $u^i:\pair{\Xi}_N \to \mathbb{C}$ solving the eikonal system.
Therefore, the coefficients
$Q^a_{\lambda,\alpha}= P^a_{\lambda,\alpha}$ also satisfy the second condition in
Equation~\eqref{elem1condition}.
That is, for any ordinary integral manifold $\iota:N \to M$ with corresponding
$P:N \to A^{(1)}\subset A \otimes V^*$, setting $Q=P|_X$ at $y \in N$ yields an
infinitesimal solution to \eqref{elem1condition}, and this solution extends to
an integral manifold $\iota^{(1)}|_D:D\to M^{(1)}$ of $\elem(\mathcal{I})$.

The submanifold $\Lambda$ in Main Theorem~\ref{foliation} is the usual
foliation for Cauchy retractions, as in Corollary~\ref{nicecoords}.
\end{proof}

\begin{proof}[Proof of Corollary~\ref{proelemC}]
In Lemma~\ref{proelem}, we are working with a linear Pfaffian ideal over an
analytic manifold with algebraic fiber of locally constant rank, all over $\mathbb{C}$, so the
Cartan--Kuranishi prolongation theorem implies that some prolongation of
$\elem(\mathcal{I})$ is involutive or empty.  By Lemma~\ref{proelem}, it is not
empty.
(See Theorem 4.2 in Chapter VI and Proposition 3.9 in Chapter VIII of
\cite{BCGGG} and the discussion therein.) 
\end{proof}

\begin{lemma}
Suppose that $(M,\mathcal{I})$ is an involutive exterior differential system.
The system $\elem(\mathcal{I})$ on $M^{(1)}$ descends to $M$ if and only if
$\elem(\mathcal{I})$ is Frobenius.
\label{descends}
\end{lemma}
\begin{proof}

Suppose that $\elem(\mathcal{I})$ descends to $M$; that is, suppose that if
$\varpi(e) = \varpi(\tilde{e})=x \in M$, then $X^1_e = X^1_{\tilde{e}}$ as
subspaces of $\mathbb{P}T_xM \otimes \mathbb{C}$; call this subspace
$\underline{X}_x$, 
which has projective dimension $n{-}L{-}1$.  Let $\omega^1, \ldots,
\omega^m$ be a coframe of $M$ near $x$ that is generic for $\mathcal{I}$ and such that 
\begin{equation}
\underline{X}^\perp =
\pair{ \omega^1, \ldots, \omega^L, \omega^{n+1}, \ldots, \omega^m}.
\label{Xbar}
\end{equation}
Using $\varpi$ to pull back this coframe to $M^{(1)}$ (and omitting writing
 $\varpi^*$), the linear Pfaffian system
$\mathcal{I}^{(1)}$ is generated by 
\begin{equation}
\begin{cases}
\theta^a = \omega^a - P^a_i \omega^i - P^a_\alpha \omega^\alpha,\\
\mathrm{d}\theta^a \equiv \pi^a_i\wedge\omega^i + \pi^a_\alpha \wedge
\omega^\alpha \mod \{ \theta^b\},
\label{New}
\end{cases}
\end{equation}
and the linear Pfaffian system $\elem(\mathcal{I})$ is generated by 
\begin{equation}
\begin{cases}
\theta^a = \omega^a - P^a_i \omega^i - P^a_\alpha \omega^\alpha,\\
\omega^i,\\
\mathrm{d}\theta^a \equiv \pi^a_\alpha \wedge
\omega^\alpha \mod \{ \theta^b, \omega^j\},\\
\mathrm{d}\omega^i \equiv \eta^i_\alpha
\wedge\omega^\alpha \mod \{ \theta^b, \omega^j\}.
\label{elemNewF}
\end{cases}
\end{equation}
with independence condition $\omega^{L+1}\wedge\cdots\wedge\omega^n \neq 0$.
This system is Frobenius if and only if $\eta^i_\alpha\wedge\omega^\alpha \equiv
\pi^a_\alpha\wedge\omega^\alpha\equiv 0$.
(The forms $\pi^a_\alpha$ and $\eta^i_\alpha$ here are not identical to those
from Equation~\eqref{elemNew}, since the coframe is different, but they play
similar roles, so we use similar notation.)
Because $\omega^i$ is basic with respect to $\varpi$, it must be that
$\eta^i_\alpha = H^i_{\alpha,j} \omega^j + H^i_{\alpha,\beta}\omega^\beta$.  Because $\elem(\mathcal{I})$
admits maximal integral manifolds by Lemma~\ref{proelem}, the independence
condition implies that the ``torsion'' terms $H^{i}_{\alpha,\beta}$ are
symmetric, so $\eta^i_\alpha\wedge\omega^\alpha \equiv 0$.
Comparing Equations~\eqref{Xbar} and \eqref{elemNewF}, we see that 
$\underline{X}^\perp=\pair{\theta^a, \omega^i} = \pair{\omega^a, \omega^i }$, so $P^a_\alpha=0$.
Differentiate and use $\mathrm{d}\omega^i \equiv 0$
to obtain $\pi^a_\alpha \wedge \omega^\alpha \equiv \mathrm{d}\omega^a$, which must vanish
because the coframe $(\omega^k)$ is basic and integral manifolds exist.

Conversely, suppose that $\elem(\mathcal{I})$, generally written in the form
of Equation~\eqref{elemNew}, is Frobenius.  It suffices to show that the
generators of $\elem(\mathcal{I})$ are basic with respect to 
$\varpi:M^{(1)} \to M$, since this is equivalent to the condition that the
Cauchy reductions of $\elem(\mathcal{I})$ contain $\ker \varpi_*$.
Of course, the 1-forms $\theta^a$ and $u^i$ are semi-basic, meaning that they
annihilate the vertical subspace $\ker \varpi_*$.  The Frobenius condition is
$\mathrm{d}\theta^a \equiv \mathrm{d}u^i \equiv 0 \mod \{ \theta^j, u^b\}$, so
these generators are basic.
\end{proof}

\begin{proof}[Proof of Main Theorems~\ref{mainthm0} and \ref{mainthmS}]
Lemma~\ref{descends} shows that \ref{t0elem} and \ref{t0desc} are equivalent.
Statements \ref{t0span} and \ref{t0sat1} are dual, and these
trivially imply statements \ref{t0elem} and \ref{t0gnf} by dimension count.
Of course, in the case of Main Theorem~\ref{mainthm0}, the Frobenius system
$\elem(\mathcal{I})$ is actually the ``irrelevant'' differential ideal, whose
integral manifolds have dimension zero. 

Suppose that statement \ref{t0elem} holds.  Then the tableau of
$\elem(\mathcal{I})$ is empty, so $\pi^a_\alpha=0$.  In particular, $v=v^\alpha
u_\alpha \in X^1_e$ implies that $v \lhk \mathrm{d}\theta^a \equiv \pi^a_\alpha
=0$, so $v \in S_e$.  This is statement \ref{t0sat1}.

Suppose that statement \ref{t0gnf} holds, and suppose that $v \in V_e$.  By
Lemmas~\ref{gnfnilp} and \ref{gnftriv}, we have that $v \in S_e$ if and only if
$v \in X^1_e$, which is \ref{t0sat1}.
\end{proof}

\begin{rmk} A Warning:  Lemma~\ref{descends}, Main Theorem~\ref{mainthm0}, and Main
Theorem~\ref{mainthmS} do \emph{not} require or imply that $\Xi$ is constant in
each fiber of $M^{(1)}$. Even if $\elem(\mathcal{I})$ is Frobenius, the
$(\pi^a_i)$ portion of the tableau \eqref{proC} may vary over $M^{(1)}$.   Conversely, even if $\Xi$ is
locally constant, the system $\elem(\mathcal{I})$ may fail to descend to $M$ if
$\mathrm{d}\theta^a \not\equiv 0 \mod \{\theta^b, \xi^j\}$.
\end{rmk}

\section{Prolonged elementary extension}
\label{proext}

Finally, we want to try to understand the case when $\mathcal{I}$ is \emph{not}
elementary, so $\elem(\mathcal{I})$ is
\emph{not} Frobenius.  
The main question is ``How can we compute $\elem(\mathcal{I})$?''  For an involutive exterior differential system with Cartan integer $\ell =
\dim \Xi + 1 > 1$
and Cartan character $s_\ell= \deg \Xi >1$, the 
(nonlinear) characteristic variety $\Xi$ is difficult to compute and
parametrize.  One might expect that selecting $L$ ``random'' elements of $\Xi$ to generate
$\pair{\Xi}$---and therefore $\elem(\mathcal{I})$---would also be difficult.
If $\mathcal{M}$ is known, then computer algebra systems allow computation of $\sat(\mathcal{M})$ using
Gr\"obner bases. 
But, it would be preferable to bypass the computation of $\Xi$ and
$\mathcal{M}$ entirely, since $X^1$ is a linear subspace of $V$ defined by
linear symbol relations, $B^\lambda_k$.
Moreover, can we hope to compute $\elem^k(\mathcal{I})$ from $B^\lambda_k$
directly for all $k \geq 2$?  
The remaining results offer a possible approach to these questions.

Recall the skewing maps $\delta$, which define tableau prolongation and are essential
to the study of involutivity via Spencer cohomology:
\begin{equation}
\begin{split}
0 \to A^{(1)} \to A \otimes V^* &\overset{\delta}{\to} W \otimes \wedge^2V^* \to
H^{2}(A)\to 0,\\
0 \to A^{(2)} \to A^{(1)} \otimes V^* &\overset{\delta}{\to} W \otimes \wedge^3V^* \to
H^{3}(A)\to 0,\\
&\vdots\\
0 \to A^{(n-1)} \to A^{(n-2)} \otimes V^* &\overset{\delta}{\to} W \otimes
\wedge^{n}V^* \to H^{n}(A)\to 0.
\end{split}
\label{spencersequence}
\end{equation}
Involutivity of the linear Pfaffian system $(M^{(1)},\mathcal{I}^{(1)})$
is equivalent to $H^{\rho}(A)=0$ for all $\rho \geq 2$.    
This condition for a formal tableau is sometimes called ``formally
integrable,'' but since our tableau comes from a linear Pfaffian system, there
is no distinction.  See Theorem 5.16 in Chapter IV of \cite{BCGGG}.

Since $X^*$ is a fixed subspace of $V^*$, let $\delta_X$ denote the restricted
skewing map that imposes symmetry only on the $\otimes^{\rho+1} X^*$ component.
The condition $\delta_X = 0$ is strictly weaker than $\delta=0$.  In
particular, $A^{(\rho)} \otimes V^*$ projects onto $A^{(\rho)} \otimes X^*$, and the
induced image
$A^{(\rho+1)}|_X$ of $A^{(\rho+1)}$ satisfies $\delta_X=0$.

Recall that the tableau of $\elem(\mathcal{I})$ is the subspace $E \subset A
\otimes X^*$ as given by Equation~\eqref{elem1condition}.  Let $\Xi_E$ denote
the characteristic variety of $E$ in $X^*$.

\begin{lemma}
Let $E^{(0)}=E \subset A \otimes X^*$ and let
$E^{(\rho)} \subset E^{(\rho-1)}\otimes X^*$ denote the $\rho$th prolongation of the
tableaux $E$ of $\elem(\mathcal{I})$.  Then, as subspaces of $A \otimes
(\otimes^{\rho} X^*)$, we have
$A^{(\rho+1)}|_{X} \subset E^{(\rho)} \subset \ker \delta_X$.
\label{dXE}
\end{lemma}
\begin{proof}
As seen in Equation~\eqref{elemNew} and Figure~\ref{figsubtab}, the tableau of
$\elem(\mathcal{I})$ is a subspace of $(U \oplus Y \oplus W) \otimes X^*$, but
the tableau conditions \eqref{elem1condition} show that the $(\eta^i_\alpha) \in
(U\oplus Y)\otimes X^*$ term depends on the $(\pi^a_\lambda) \in A$
term when using our adapted coframe \eqref{ucoframe}.  Therefore, we may
consider the tableau of $\elem(\mathcal{I})$ to be the subspace $E \subset A \otimes
X^*$ specified by those conditions.  
The first condition in \eqref{elem1condition} is $\delta_X Q = 0$, and the
independence condition imposes symmetry over $\otimes^\rho X^*$ for $\rho \geq 1$.

For any $P \in A^{(1)}$, the proof of Lemma~\ref{proelem} says that the eikonal system
forces $P|_X \in E$ via the restriction of $P \in A^{(1)} \subset A \otimes
V^*$ to $A \otimes X^*$.  Involutivity, the characteristic variety, and the
eikonal system are all preserved by prolongation of $A$, so this containment is
preserved as well.
\end{proof} 

One problem with Conjecture~\ref{invconj} is that $E$ is fairly annoying to
compute; specifically, the $\eta^i_\alpha$ terms in Equation~\eqref{elemNew} and
Figure~\ref{figsubtab} depend on the local coframe chosen on $M^{(1)}$.  This
dependency can be ignored if the EDS $(M^{(1)},\mathcal{I}^{(1)})$ arises from a ``local
PDE in jet-space'' with the additional condition that the span of the
characteristic variety is locally constant in the coordinates $\mathrm{d}x^1, \ldots, \mathrm{d}x^n$.  
Then we can take the adapted coframe $(u^k)$ to be closed, giving $\eta^i_\alpha =0$.  

Let $\dot{A}$ denote the formal
tableau\footnote{It is ``formal'' in the sense that it did not arise \emph{a
priori} from a particular EDS.} obtained the projection of $A$ to $W\otimes X^*$.  Then $\dot{A}$ is
given by an exact sequence 
\begin{equation} 
\emptyset \to \dot{A} \to W \otimes X^* \overset{\dot{\sigma}}{\to} \dot{A}^\perp \to \emptyset 
\label{symtabX}
\end{equation}
induced by the sequence~\eqref{symtab}.  Since we have a good description of
$A^\perp$, it is easy to write 
\begin{equation}
\begin{split}
\dot{A} 
&= \left\{ \pi|_X,\ \pi \in A \right\} = \left\{  \pi^a_\alpha (w_a \otimes u^\alpha),\ \pi
\in A \right\}\\
&= \left\{ B^{a,\lambda}_{\alpha,b}\pi^a_\lambda (w_a \otimes u^\alpha),\ \pi \in A\right\}\\
&= \pair{  B^{a,\lambda}_{\alpha,b} (w_a \otimes u^\alpha) }
\end{split}
\end{equation}
and 
\begin{equation}
\begin{split}
\dot{A}^\perp &=
\left\{ K=K^\alpha_a w^a \otimes u_\alpha \in (W \otimes X^*)^* ~:~ 
K^\alpha_a B^{a,\lambda}_{\alpha,b} =0,\ \forall\ b \leq s_\lambda  \right\}.
\end{split}
\end{equation}
The characteristic variety of $\dot{A}$ is 
\begin{equation}
\dot{\Xi} = \left\{ \xi \in X^* ~:~  \ker (K^\alpha_a \xi_\alpha w^a) \neq
0\ \forall K \in \dot{A}^\perp\right\}.
\end{equation}
The skewing map on $X^*$ defines a formal prolongation of $\dot{A}$, 
\begin{equation}
0 \to \dot{A}^{(1)} \to \dot{A} \otimes X^* \overset{\delta_X}{\to} W \otimes \wedge^2 X^* \to
H^{2}(\dot{A}) \to 0.
\end{equation}
The characteristic variety of $\dot{A}^{(1)}$ is 
\begin{equation}
\begin{split}
\dot{\Xi}^{(1)} &= \left\{ \xi \in X^* ~:~  \exists \pi \in A,\ \delta_X(\pi|_X
\otimes \xi) = 0 \right\}\\
 &= \left\{ \xi \in X^* ~:~  \exists \pi \in A,\ \delta_X(\pi
\otimes \xi) = 0 \right\}.
\end{split}
\end{equation}

Compare the next lemma to Corollary~\ref{cor:restrict} and Theorem A in
\cite{Guillemin1968}, which is much harder
due to a looser notion of involutivity for formal tableaux.
\begin{lemma}
If $(M,\mathcal{I})$ is involutive, then $H^{\rho}(\dot{A})=0$ for all $\rho \geq 2$.
\label{killspencer}
\end{lemma}
\begin{proof}
If $(M,\mathcal{I})$ is involutive, then $H^{\rho}(A)=0$ for all $\rho \geq
2$.  The maps $A \to \dot{A}=A|_X$ 
and $V^* \to X^*$ are surjective and commute with $\delta$, so the same applies
to $\dot{A}$.
\end{proof}
Lemma~\ref{killspencer} says that, if we can associate $\dot{A}$ with a linear
Pfaffian exterior differential system, then that system is involutive.  This is
useful in the local PDE case where $\pair{\Xi}$ is locally constant, for then $E=\dot{A}$ because $u^\lambda = \mathrm{d}x^\lambda$. Conjecture~\ref{invconj} and Main Theorem~\ref{mega-foliation} follow immediately.

Even in the general case, our only hope for a general result regarding
$\elem^2(\mathcal{I})$ is if $(\eta^i_\alpha)$ is determined by
$(\pi^a_\alpha)$, so $\dot{A}$ is still worth studying.

\begin{lemma}
As subsets of $X^*$, we have $\Xi^{(1)}_E \subset \Xi_E \subset \dot{\Xi}^{(1)}
\subset \dot{\Xi}$.
\label{elemchar}
\end{lemma}
\begin{proof}
The left-most and right-most inclusions are standard; prolongation can
only increase the characteristic ideal of a tableau.  (See the discussion leading
to statement (79) in Chapter V of \cite{BCGGG}.)

Suppose that $\xi \in \Xi_E$.
Then there exists $\pi \in A$ such that $\pi \otimes \xi$ is a
rank-one element of $E \subset A \otimes X^*$.  The first
condition in \eqref{elem1condition} means $\delta_X Q = 0$, so $\pi|_X \in \dot{A}$ and $(\pi|_X)\otimes \xi \in
\dot{A}^{(1)}$.   This is rank-one, so $\xi$ lies in the characteristic variety
of $\dot{A}^{(1)}$.  

To see the weaker inclusion $\Xi_E \subset \dot{\Xi}$, consider the generating 2-forms \eqref{elemNew} and
Figure~\ref{figsubtab}.  If the combined matrix $\left( \pi^a_\alpha w_a +
\eta^i_\alpha u_i \right)\otimes u^\alpha \in E$ is rank-one, then the upper
matrix $\pi^a_\alpha (w_a \otimes u^\alpha) \in \dot{A}$ is rank-one over the
same fiber.
\end{proof}
If we knew these were involutive, then the degree of $\xi$ in the variety would
be seen to fall by a constant: the nullity of the projection $\pi\mapsto \pi|_X$.

The next conjecture would help establish a general equivalence between
$\dot{A}$ and $E$.

\begin{conj}
Suppose that $(M,\mathcal{I})$ is an involutive exterior differential system. 
The characteristic sheaf of $E$ equals the characteristic sheaf of $\dot{A}$.
\label{rankconj}
\end{conj}
A weaker version would suffice if we merely want to compute
$\elem^2(\mathcal{I})$, regardless of its involutivity.
\begin{conj}
Suppose that $(M,\mathcal{I})$ is an involutive exterior differential system. 
Then $\langle\dot{\Xi}\rangle=\pair{\Xi_E}$ as subspaces of $X^*$.
\label{rankconj1}
\end{conj}

\begin{rmk}
On the question of involutivity for $\elem(\mathcal{I})$:
If one were to consider Conjecture~\ref{invconj} for a formal
tableaux $A$ (as opposed to a tableau coming from a torsion free involutive
exterior differential system) then studying $\dot{A}=A|_X$ itself is very difficult.  Unfortunately, the only known result
on involutivity of sub-tableaux is the theorem of \cite{Guillemin1968}, which is generalizes \cite{Guillemin1968} and is restated here as Corollary~\ref{cor:restrict}.  This theorem applies to non-characteristic sub-tableaux like $A|_U$, but our
sub-tableaux $A|_X$ is defined to be \emph{maximally characteristic}!
\end{rmk}

\section{Parabolic Examples}\label{parabolic}
The prototypical example of a non-elementary system is the 1-dimensional heat equation on $y(t,x)$,
\begin{equation}
\partial_t  y =  \partial_x^2 y.
\label{eq:heat1}
\end{equation}  
On the manifold $M
\cong \mathbb{R}^{7}$ with local coordinates $(t, x, y, p_t, p_x, P_{tt},
P_{tx})$, consider the differential ideal generated by the contact
1-forms 
\begin{equation}
\begin{cases}
\Upsilon^0 = \mathrm{d}y - p_t \mathrm{d}t - p_x \mathrm{d}x,& \\
\Upsilon^1 = \mathrm{d}p_t - P_{tt} \mathrm{d}t - P_{tx} \mathrm{d}x,& \\
\Upsilon^2 = \mathrm{d}p_x - P_{tx} \mathrm{d}t - p_{t} \mathrm{d}x,&
\end{cases}
\end{equation}
and their derivative 2-forms
\begin{equation}
\mathrm{d}
\begin{bmatrix}
\Upsilon^1 \\ \Upsilon^2 \\ \Upsilon^0
\end{bmatrix}
\equiv
-
\begin{bmatrix}
\mathrm{d}P_{tt} & \mathrm{d}P_{tx} \\
\mathrm{d}P_{tx} & 0 \\
0 & 0 
\end{bmatrix}
\wedge
\begin{bmatrix}
\mathrm{d}t \\ \mathrm{d}x
\end{bmatrix}
-
\begin{bmatrix}
0 \\ P_{tt} \\ 0 \end{bmatrix}
\mathrm{d}t \wedge \mathrm{d}x,\ \mod \Upsilon^1,\Upsilon^2,\Upsilon^0.
\end{equation}
Absorb torsion by setting 
$\kappa^1 = - \mathrm{d}P_{tt}$ and $\kappa^2 = - \mathrm{d}P_{tx} + P_{tt}\mathrm{d}x$.  
Then 
\begin{equation}
\mathrm{d}
\begin{bmatrix}
\Upsilon^1 \\ \Upsilon^2 \\ \Upsilon^0
\end{bmatrix}
\equiv
-
\begin{bmatrix}
\kappa_1 & \kappa_2 \\
\kappa_2 & 0 \\
0 & 0 
\end{bmatrix}
\wedge
\begin{bmatrix}
\mathrm{d}t \\ \mathrm{d}x
\end{bmatrix}.
\label{eq:heat1tab}
\end{equation}

At this point, one is tempted to declare ``Look!  The rank-one cone is given by
$(\kappa^2)^2=0$, so the characteristic variety is just $\mathrm{d}t$, with
multiplicity two.''  However, this is imprecise, since the characteristic variety is interpreted properly as a sub-variety of the canonical bundle $V^*$ over $M^{(1)}$.  The prolongation $M^{(1)}$ must be constructed to apply the Main Theorems.

The tableau expressed by Equation \eqref{eq:heat1tab} has $\ell=1$ and $s=s_1=2$, so the fiber of the prolongation $M^{(1)} \to M$ is parametrized by local coordinates $U_1$ and $U_2$.  This can be seen by considering the general Grassmannian contact relations 
\begin{equation}
\begin{cases}
\kappa_1|_e = g_{1,1}(e) \eta^1|_e + g_{1,2}(e) \eta^2|_e,\\
\kappa_2|_e = g_{2,1}(e)\eta^1|_e + g_{2,2}(e)\eta^2|_e.
\end{cases}
\end{equation}
For any $e \in M^{(1)}$, the conditions $\mathrm{d}\Upsilon^1|_e=\mathrm{d}\Upsilon^2|_e =0$ imply $g_{2,2}(e)=0$ and $g_{1,2}(e)=g_{2,1}(e)$.  Therefore, we take $U_1 = g_{1,1}$ and $U_2=g_{1,2}$ as fiber coordinates on the 9-dimensional submanifold $M^{(1)} \subset \Gr_2(TM)$.

Therefore, for any $e \in M^{(1)}$, we establish a basis of $T_eM^{(1)}$ by setting 
\begin{equation}
\begin{cases}
u^1 = \mathrm{d}t,\\
u^2 = \mathrm{d}x,\\
u^3 = \Upsilon^0 = \mathrm{d}y - p_t \mathrm{d}t - p_x \mathrm{d}x,\\
u^4 = \Upsilon^1 = \mathrm{d}p_t - P_{tt} \mathrm{d}t - P_{tx} \mathrm{d}x, \\
u^5 = \Upsilon^2 = \mathrm{d}p_x  - P_{tx} \mathrm{d}t, \\
u^6 = \kappa^1 - U_1(e) \mathrm{d}t - U_2(e) \mathrm{d}x, \\
u^7 = \kappa^1 - U_1(e) \mathrm{d}t - U_2(e) \mathrm{d}x, \\
u^8 = \mathrm{d}U_1, \\
u^9 = \mathrm{d}U_2, 
\end{cases}
\end{equation}
and omitting writing the pull-back of $M^{(1)} \to M$ on the right-hand side.  
Decomposing $T_eM^{(1)}\otimes \mathbb{C} = V_e^* + W_e + A_e$ according to this basis,  
we see that the system 
$\mathcal{I}^{(1)}$ is generated by the 1-forms $u^3,\ldots,u^7$ spanning $W_e$ and the 2-forms
\begin{equation}
\mathrm{d}
\begin{bmatrix}
u^6\\ u^7
\end{bmatrix}
\equiv 
\begin{bmatrix}
u^8 & u^9 \\
u^9 & 0 
\end{bmatrix}\wedge
\begin{bmatrix}
u^1 \\ u^2
\end{bmatrix},\ \mod u^2, \ldots, u^7.
\label{eq:heat1tab1}
\end{equation}
It is easy to check involutivity with either Cartan's test or Theorem~\ref{thm-gnf+}.

Now it is sensible to say that $\C_e \subset A_e$ is given by the condition $(u^9)^2=0$, so $\Xi_e \subset V^*_e$ is spanned by $u^1$.  
Clearly, condition \ref{t0span} of Main Theorem~\ref{mainthm0} fails.
The failure of condition \ref{t0gnf} in Main Theorem~\ref{mainthm0} is evident,
since $\left(B(u^1)(u_2)\right)^2=0$.  The failure of condition \ref{t0sat1}
also evident:  as in Figure~\ref{figsubtab}, the system $\elem(\mathcal{I})$ is differentially generated by the 1-forms
$u^1, u^3, \ldots, u^7$ but contains the non-trivial 2-form $\mathrm{d} u^6
\equiv  u^9 \wedge u^2$ modulo $u^1, u^3, \ldots, u^7$.
Finally, condition \ref{t0desc} fails, because the generators $u^6$ and $u^7$ and the non-trivial derivative $\mathrm{d}u^6$ vary with $e \in M^{(1)}$.  
Main Theorem~\ref{foliation} takes the form of Corollary~\ref{cor:heat1}.
\begin{cor}
Every ordinary 2-dimensional integral manifold of \eqref{eq:heat1tab1} is
foliated by 1-dimensional hypersurfaces satisfying $u^1=0$.
\label{cor:heat1}
\end{cor}
Regarding Conjecture~\ref{proelemC},  note that $\elem(\mathcal{I})$ is
involutive and elementary with $s_1=1$ and that $\elem^2(\mathcal{I})$ is the
irrelevant Frobenius ideal.  This reflects the existence of coordinates
$\mathrm{d}t$ and $\mathrm{d}x$ on solutions.  Establishing the existence of
these coordinates may seem pointless because we assumed their existence to write
the system \eqref{eq:heat1tab} originally; however, note that Corollary~\ref{cor:heat1}
yields such coordinates for \emph{any} differential system presented in a
coframe as \eqref{eq:heat1tab1}, even if $u^1$ and $u^2$ are not closed
on the space where the ideal is defined.

Extending the analogy between non-elementary systems and parabolic systems to higher dimensions is subtle.  Consider the 2-dimensional heat equation on $y(x^1,x^2,x^3)$,
\begin{equation}
\partial_1  y =  \left( \partial_{2}^2+ \partial_{3}^2\right) y.
\label{eq:heat2}
\end{equation}  
The ideal on jet space is generated by the 1-forms
\begin{equation}
\begin{cases}
\Upsilon^0=\mathrm{d}y-p_1\mathrm{d}x^1-p_2\mathrm{d}x^2-p_3\mathrm{d}x^3\\
\Upsilon^1=\mathrm{d}p_1-P_{11}\mathrm{d}x^1-P_{12}\mathrm{d}x^2-P_{13}\mathrm{d}x^3\\
\Upsilon^2=\mathrm{d}p_2-P_{12}\mathrm{d}x^1-P_{22}\mathrm{d}x^2-P_{23}\mathrm{d}x^3\\
\Upsilon^3=\mathrm{d}p_3-P_{13}\mathrm{d}x^1-P_{23}\mathrm{d}x^2-\left(p_1-P_{22}\right)\mathrm{d}x^3
\end{cases}
\end{equation}
and their derivative 2-forms.
Proceeding to absorb torsion and change bases, we arrive at a tableau of the form
\begin{equation}
\begin{bmatrix}
\pi^1_1 & \pi^1_2 & -\pi^2_2 \\
\pi^2_1 & \pi^2_2 & \pi^1_2 \\
\pi^3_1 & \pi^2_1 & \pi^1_1
\end{bmatrix},
\end{equation}
which $\ell=2$, $s_1=3$, and $s_2=2$.  It has symbol maps
\begin{equation}
B(u^1)(u_2) = \begin{bmatrix} 0 & 0 & 0  \\ 0 & 0 & 0\\ 0 & 1 & 0 \end{bmatrix},\ 
B(u^1)(u_3) = \begin{bmatrix} 0 & 0 & 0  \\ 0 & 0 & 0\\ 1 & 0 & 0 \end{bmatrix},
\end{equation}
and 
\begin{equation}
B(u^2)(u_3) = \begin{bmatrix} 0 & -1 & 0  \\ 1 & 0 & 0\\ 0 & 0 & 0 \end{bmatrix}.
\end{equation}
Note that this system is elementary over $\mathbb{C}$; ignoring multiplicity,
its characteristic variety is comprised of points of the form $[v_1: v_2: \pm i v_2]$,
which is two independent $\mathbb{CP}^1$'s.

However, it is also clear that (in some imprecise sense) this is non-elementary
over $\mathbb{R}$ because $B(u^1)(u_3)$ is nilpotent and $B(u^2)(u_3)$ has no non-zero \emph{real} eigenvalues,
In other words,  $B(\xi)(v) = \xi_1 B(u^1)(v) + \xi_2 B(u^2)(v)$ has 
non-zero \emph{real} eigenvalues on $\Wu(\xi)^-$ if and only if $u^3(v)=0$.
So there is a 1-dimensional subspace of $V$ on which the conditions of Main
Theorem~\ref{mainthm0} fails.   On the other hand, allowing $v_2$ to be imaginary,
then the same reasoning applies with $u^2(v)=0$. 
Pursuit of the real case from this perspective is a compelling subject for
future work, as it may lead to general results for parabolic PDEs.

\section{Artificial Examples}\label{sec:art}
To construct new non-elementary involutive tableaux, we can apply the approach of \cite{Smith2014a}, as expressed by 
Theorem~\ref{thm-gnf+} and the preferred decomposition $V^* = U^* \oplus Y^* \oplus X^*$ from Section~\ref{gnf}.
To illustrate this, let us construct an involutive tableau with $(\ell,L,\nu,n) = (3,4,5,5)$ and 
$(s_1,s_2,s_3) = (3,2,2)$.  Because $s_1=3$, we may as well take $r=3$ to avoid writing zero rows in the tableau.
Because $\nu=n$, the tableau takes the following form, where all columns are
linearly independent,
\begin{equation}
\pi = 
\begin{bmatrix}
\pi_1 & \pi_4  &  \pi_6  &  ?  & ? \\
\pi_2 & \pi_5  &  \pi_7  &  ?  & ? \\
\pi_3 & ?  &  ?  &  ?  & ?  
\end{bmatrix}.
\end{equation}
We can build an example tableau by choosing how the remaining entries depend on $\pi_1, \ldots, \pi_7$; that is, we may choose the $r \times r$ matrices $B^\lambda_k = B(u^\lambda)(u_k)$ in $\End(W)$.
To facilitate computation, write 
\begin{equation}
\left(B^\lambda_k\right)
=\begin{pmatrix}
I_3 & B^1_2 & B^1_3 & B^1_4 & B^1_5 \\
      & I_2 & B^2_3 & B^2_4 & B^2_5 \\
      &       & I_3 & B^3_4 & B^3_5 
\end{pmatrix}.
\label{eq:Barray}
\end{equation}

Because $s_2=s_3=2$, the matrices $B^2_k$ and $B^3_k$ are zero outside the upper-left $2 \times 2$ block for all $k$.   Also, $B^1_2$, $B^1_3$, and $B^2_3$ have zeros in the upper-left $2 \times 2$ block.

The condition $L=4$ implies that $B^\lambda_5$ is nilpotent on
$\Wu^1(u^\lambda)$ for all $\lambda=1,2,3$.  The non-trivial $2 \times 2$ nilpotent matrices are 
\begin{equation}
\begin{pmatrix} 0 & 1 \\ 0 & 0 \end{pmatrix},\
\begin{pmatrix} 0 & 0 \\ 1 & 0 \end{pmatrix},\ \text{and}\ 
\begin{pmatrix} 1 & q^{-1} \\ -q & -1 \end{pmatrix}.
\end{equation}
Let's assume that $B^3_5$ is the first of these forms and that $B^2_5$ is the third of these forms.
The nilpotency of $B^1_5$ on $\Wu^1(u^1)$ is examined below.

Theorem~\ref{thm-gnf+} implies that $B^2_4$ shares a Jordan form with $B^2_5$, so 
$B^2_4 = r_1 I_2 + p_1 B^2_5$.  Similarly, $B^3_4$ takes the form $r_2 I_2 +
p_2 B^3_5$.  Moreover, the condition $B^2_4B^3_5 - B^2_5B^3_4=0$ implies that
$r_1=r_2=0$ and $p_1=p_2=p$.  Therefore, \eqref{eq:Barray} has 
\begin{align}
B^2_4 &= \begin{bmatrix} 0 & p &  \\ 0 & 0 &\\\phantom{M}  &\phantom{M}  &\phantom{M}  \end{bmatrix},& 
B^2_5 &= \begin{bmatrix} 0 & 1 &  \\ 0 & 0 &\\\phantom{M}  &\phantom{M}  &\phantom{M}  \end{bmatrix}, \\
B^3_4 &= \begin{bmatrix} p & \frac{p}{q} &\\ -pq & -p &\\\phantom{M}  &\phantom{M}  & \phantom{M} \end{bmatrix},& 
B^3_5 &= \begin{bmatrix} 1 & \frac{1}{q} &\\ -q & -1 &\\\phantom{M}  &\phantom{M}  & \phantom{M} \end{bmatrix}
\end{align}

Next, we know that $\left(B^1_2 B^2_3 - B^1_3 B^2_2\right)^a_b=0$ for $a>s_2 = 2$.  Because $B^2_3=0$ and $B^2_2=I_2$, the matrix $B^1_3$ must be of the form
\begin{equation}
B^1_3 = \begin{bmatrix}
0 & 0 & 0 \\
0 & 0 & 0 \\
0 & 0 & z_3 
\end{bmatrix}.
\end{equation}
Similarly, $\left(B^1_2 B^3_3 - B^1_3 B^3_2\right)^a_b=0$ for $a>s_2 = 2$.  Because $B^3_2=0$ and $B^3_3=I_2$, the matrix $B^1_2$ must be of the form
\begin{equation}
B^1_2 = \begin{bmatrix}
0 & 0 & 0 \\
0 & 0 & 0 \\
0 & 0 & z_2 
\end{bmatrix}.
\end{equation}

Next, $\left(B^1_2 B^2_4 - B^1_4 B^2_2\right)^a_b=0$ and
$\left(B^1_2 B^2_5 - B^1_5 B^2_2\right)^a_b=0$ 
for $a>s_2 = 2$.  Because $0=B^1_2B^2_4=B^1_2B^2_5$, it must be that 
\begin{align}
B^1_4 &= \begin{bmatrix}
? & ? & ? \\
? & ? & ? \\
0 & 0 & z_4
\end{bmatrix},&
B^1_5 &= \begin{bmatrix}
? & ? & ? \\
? & ? & ? \\
0 & 0 & z_5
\end{bmatrix}.
\end{align}
Next, note that $[1:z_2:z_3:z_4:z_5] \in \Xi$, so $L=4$ implies $z_5=0$.

The conditions $\left(B^1_4 B^2_5 - B^1_5B^2_4\right)^a_b=0$ 
and $\left(B^1_4 B^3_5 - B^1_5B^3_4\right)^a_b=0$ for $a>s_4=0$ imply that 
\begin{align}
B^1_4 &= \begin{bmatrix}
px_1 & px_2 & y_3 \\
px_3 & px_4 & y_4 \\
0 & 0 & z_4
\end{bmatrix},&
B^1_5 &= \begin{bmatrix}
x_1 & x_2 & y_1 \\
x_3 & x_4 & y_2 \\
0 & 0 & 0
\end{bmatrix}.
\end{align}

Assuming that the parameters $z_2$ and $z_3$ are non-zero, the space
$\Wu^1(u^1)$ is comprised of vectors with third component zero.  Because one is
a scale of the other on $\Wu^1(u^1)$, $B^1_5$ and $B^1_4$ already commute on
that subspace.  To make $B^1_5$ nilpotent on $\Wu^1(u^1)$, let's assume 
$x_1=g$, $x_2=\frac{g}{h}$, $x_4=-g$, and $x_3=-hg$.

Finally, the condition $\left(B^1_4 B^1_5 - B^1_5 B^1_4\right)^a_b=0$ for
$a>s_3=0$ implies compatible Jordan forms on all of $W=\Wu^-(u^1)$.  
An example solution is to force $B^1_4$ and $B^1_5$ to be block lower-triangular by setting
$y_1,y_2,y_3,y_4$ zero.

Hence, an example non-elementary tableau with the desired dimensions is
{\small
\begin{equation*}
\begin{bmatrix}
\pi_{1} & \pi_{4} & \pi_{6} & p \left(g\pi_{1} +
\frac{g}{h}\pi_{2} + \pi_{5} + \pi_{6} + \frac{1}{q}\pi_{7}\right) & g\pi_{1}+ \frac{g}{h}\pi_{2}   + \pi_{5} + \pi_{6}
+ \frac{1}{q}\pi_{7} \\
\pi_{2} & \pi_{5} & \pi_{7} & p\left(-gh\pi_{1} -g\pi_{2}- p\pi_{6} 
 - \pi_{7}\right) & -gh\pi_{1} -g\pi_{2} - q\pi_{6}   - \pi_{7} \\
\pi_{3} & \pi_{3} z_{2} & \pi_{3} z_{3} & \pi_{3} z_{4} & 0
\end{bmatrix}
\end{equation*}}
The parameters $p,q,h,g,z_2,z_3,z_4$ could be taken as functions on $M^{(1)}$.
(Note that $z_4 =0$ if and only if $L$ falls to $\ell=3$.)
In principle, one could construct all involutive tableaux this way.



\newpage

\section{Discussion}
The main theorems are direct observations using established techniques in exterior
differential systems.  They fill a significant gap in the literature, but in
retrospect they may not be surprising to experts who have manipulated
tableaux in many examples.  
Notably,
Cartan encountered many of these phenomena in \cite{Cartan1911}---particularly
the example beginning with paragraph 22---but apparently he did not pursue them elsewhere.
To my knowledge, that is the only appearance of any similar statement in the
literature.

As to the coinage ``elementary,'' several other names also seem appropriate.
One might call these systems names like ``semi-simple,'' ``non-parabolic,'' or
``primitive.'' But, I believe ``semi-simple'' is premature without a generalized
Levi decomposition theorem, ``non-parabolic'' is misleading without a generalized
regularity theorem, and ``primitive'' would convolute the intricate
relationship between involutive EDS and Lie pseudogroups.

However, it does seem reasonable to expect that there would be a form 
of parabolic regularity to certain non-elementary systems where $n-L=1$.
Involutivity of the eikonal system of $\pair{\Xi}$ should guarantee the
existence of a time variable corresponding to the vector subspace $X$ of
nilpotents, like in Corollary~\ref{determined}.  Of course, as seen in Section~\ref{parabolic}, these results would need to be loosened from
$\mathbb{C}$ to $\mathbb{R}$ to be useful for analysis.

I do \emph{not} have a strong sense of whether the Conjectures are
actually true. They seem to hold on examples I have constructed by hand using
Theorem~\ref{thm-gnf+} as in Section~\ref{sec:art}, but
it is difficult to build toy systems that have sufficiently complicated
characteristic varieties after elementary reduction.  I encourage you to sift
through your favorite non-elementary EDS/PDEs for examples where
Conjecture~\ref{rankconj} holds or fails.  Where it holds, it suggests that all
solutions of the EDS can be found through a canonical sequence of
reductions.  Such a structure would provide a solvability criterion, allowing a
decomposition theorem for EDS into a sequence elementary or Frobenius systems,
as in Main Theorem~\ref{mega-foliation}.  If it fails, then the microlocal
analysis of characteristic varieties contains further mysteries that must be
similarly fascinating.

\newpage
\linespread{1.0} 
\normalfont \selectfont
\providecommand{\etalchar}[1]{$^{#1}$}
\providecommand{\bysame}{\leavevmode\hbox to3em{\hrulefill}\thinspace}
\providecommand{\noopsort}[1]{}
\providecommand{\mr}[1]{\href{http://www.ams.org/mathscinet-getitem?mr=#1}{MR~#1}}
\providecommand{\zbl}[1]{\href{http://www.zentralblatt-math.org/zmath/en/search/?q=an:#1}{Zbl~#1}}
\providecommand{\jfm}[1]{\href{http://www.emis.de/cgi-bin/JFM-item?#1}{JFM~#1}}
\providecommand{\arxiv}[1]{\href{http://www.arxiv.org/abs/#1}{arXiv~#1}}
\providecommand{\doi}[1]{\url{http://dx.doi.org/#1}}
\providecommand{\MR}{\relax\ifhmode\unskip\space\fi MR }
\providecommand{\MRhref}[2]{%
  \href{http://www.ams.org/mathscinet-getitem?mr=#1}{#2}
}
\providecommand{\href}[2]{#2}

\end{document}